\documentclass[36pt, a4paper]{article}
\usepackage{amsthm,amsmath,amsfonts,amssymb, hyperref,authblk}
\usepackage{blindtext}
\usepackage{lipsum}                     
\usepackage{xargs}                      
\usepackage[pdftex,dvipsnames]{xcolor}  
\usepackage[colorinlistoftodos,prependcaption,textsize=tiny]{todonotes}
\newcommandx{\unsure}[2][1=]{\todo[linecolor=red,backgroundcolor=red!25,bordercolor=red,#1]{#2}}
\newcommandx{\change}[2][1=]{\todo[linecolor=blue,backgroundcolor=blue!25,bordercolor=blue,#1]{#2}}
\newcommandx{\info}[2][1=]{\todo[linecolor=OliveGreen,backgroundcolor=OliveGreen!25,bordercolor=OliveGreen,#1]{#2}}
\newcommandx{\improvement}[2][1=]{\todo[linecolor=Plum,backgroundcolor=Plum!25,bordercolor=Plum,#1]{#2}}
\newcommandx{\thiswillnotshow}[2][1=]{\todo[disable,#1]{#2}}
\usepackage{xcolor}
\usepackage{bm}
\usepackage{cases}
\usepackage{booktabs} 
\usepackage{array}
\usepackage{siunitx}
\usepackage{scrtime}
\usepackage{mwe} 
\usepackage{caption}
\usepackage{graphics}
\usepackage{caption}
\usepackage{float}
\usepackage{stmaryrd}
\usepackage{xparse}
\usepackage{booktabs}
\usepackage[noend]{algorithmic} 
\usepackage{algorithm}
\usepackage{stmaryrd}
\usepackage[figuresright]{rotating}

\usepackage{vmargin}
\setmarginsrb{4cm}{1cm}{2cm}{1cm}{1cm}{1cm}{2cm}{4cm}

\renewcommand{\le}{\leqslant}

\renewcommand{\ge}{\geqslant}

\renewcommand{\div}{\operatorname{div}}

\newcolumntype{L}[1]{>{\raggedright\arraybackslash}p{#1}}

\everymath{\displaystyle}

\newtheorem{thm}{Theorem}[section]
\newtheorem{lemma}[thm]{Lemma}
\newtheorem{proposition}[thm]{Proposition}

\newtheorem{e-definition}[thm]{Definition}
\newtheorem{remark}[thm]{\bf Remark\/}

\numberwithin{equation}{section}

\makeatletter
\newcommand\appendix@section[1]{%
\refstepcounter{section}%
\orig@section*{Appendix \@Alph\c@section: #1}%
}
\let\orig@section\section
\g@addto@macro\appendix{\let\section\appendix@section}
\makeatother

 \title{Non-separable Mean Field Games for Pedestrian Flow: Generalized Hughes Model}     

 \author{Mohamed Ghattassi \thanks{NYUAD Research Institute, New York University Abu Dhabi, PO Box 129188, Abu Dhabi, United Arab Emirates, {\sf mg6888@nyu.edu}} \quad Nader Masmoudi \thanks{Department of Mathematics, New York University in Abu Dhabi, Saadiyat Island, P.O. Box 129188, Abu Dhabi, United Arab Emirates-- Courant Institute of Mathematical Sciences, New York University, 251 Mercer Street, New York, NY 10012, USA, {\sf nm30@nyu.edu}}}

\date{\today}

\begin{document}                  
\maketitle   
 \tableofcontents

\begin{abstract}


In this paper, we present a new generalized Hughes model designed to intelligently depict pedestrian congestion dynamics, allowing pedestrian groups to either navigate through or circumvent high-density regions. First,  we describe the microscopic settings of the model. The corresponding optimization problems are deterministic and can be formulated by a closed-loop model predictive control strategy. This microscopic setup leads in the mean-field limit to the generalized Hughes model which is a class of non-separable mean field games  system, i.e., Fokker-Planck equation and viscous Hamilton-Jacobi Bellman equation are coupled in a forward-backward structure. We give an overview on the mean field games in connection to our intelligent fluid model. Therefore, we show the existence of weak solutions to the generalized Hughes model and analyze the vanishing viscosity  limit of weak solutions. Finally, we illustrate the generalized Hughes model with various numerical experiments.

   \vspace{0.2cm}
\end{abstract}

\section{Introduction}
In recent years, the study of pedestrian dynamics has become a focal point for numerous researchers in various scientific disciplines. Originating from the realm of traffic engineering, the movement and behavior of pedestrian crowds have garnered increasing attention from the mathematics community. For an in-depth exploration of modeling and analytical challenges in this domain, please refer to \cite{bellomo2022towards,bellomo2023human}. Modeling pedestrian dynamics can be classified crowd by three possible modeling scales (i.e. microscopic, mesoscopic, macroscopic) as follows:
\begin{itemize}
\item  \textbf{Microscopic (individual-based) Models:} force-based models such as the social force model, see \cite{helbing1995social,helbing2000simulating,aurell2018mean,aurell2020behavior,li2020approach}, cellular automata approaches, \cite{burstedde2001simulation,burger2011continuous}.
\item  \textbf{Mesoscopic (Kinetic) Models:} The interactions between individuals are modeled by the collisions where the ideas are initially used on the gas kinetics; see \cite{bellomo2013microscale,liao2023kinetic,bellomo2023human} and references therein.
\item  \textbf{Macroscopic (hydrodynamic) Models:} The macroscopic model uses an Eulerian-type hydrodynamic description of a system of crowd dynamics, see \cite{lachapelle2011mean,bellomo2015multiscale,borzi2020fokker,liang2021continuum,bellomo2023human}, and some variants of the popular Hughes model\cite{huang2009revisiting,burger2014mean,carrillo2016improved,fischer2020micro,herzog2020optimal,pietschmann2023numerical}.
\end{itemize}

In 2002, Roger Hughes proposed a macroscopic model for crowd dynamics in which individuals seek to minimize their travel time by avoiding regions of high density. The celebrated Hughes model describes fast exit and evacuation scenarios, where a group of people wants to leave a domain $\Omega \subset\mathbb{R}^{2}$ with one or several exits/doors and/or obstacles as fast as possible. The macroscopic model for pedestrian dynamics in \cite{Hughes2002continuum} is based on a continuity equation (describing the evolution of crowd density) and an Eikonal equation (giving the shortest weighted distance to an exit). The model is given by
\begin{align}
\partial_{t}\rho -div(\rho f^{2}(\rho)\nabla \phi )&=0 \label{FluidEq}\\
|\nabla\phi| &=\frac{1}{f(\rho)}\label{EikEq}\\
\rho(t=0,x)&= \rho_{in}(x),\label{FluidIn}
\end{align}
where $x\in \Omega$ denotes the position in space, $t\in (0,T]$, $T>0$, the time variable  and $\nabla$ the gradient with respect to the space variable $x$. The function $\rho$ corresponds to the pedestrian density, $\rho_{in}$ is the initial pedestrian density and $\phi$ the weighted shortest distance to the exit. $f$ is a function introducing saturation effects such as $f(\rho)=\rho_{m}-\rho$, where $\rho_{m}$ corresponds to the maximum scaled pedestrian density. Hughes model \eqref{FluidEq}-\eqref{FluidIn} is supplemented with different boundary conditions for the
walls and the exits.  We assume that the boundary $\partial \Omega$ of our domain is subdivided into three parts: inflow $\Gamma$, outflow $\Sigma$ and insulation $\Gamma_{a}=\partial \Omega / (\Gamma \cup\Sigma)$ with $\Gamma \cap \Sigma =\emptyset$. We assume that particles enter the domain $\Omega$ on $\Gamma$ with boundary pedestrian density $\rho_{b}$. On the remaining part of the boundary $\Gamma_{a}$, we impose no flux conditions 
\begin{align}
  \rho &= \rho_{b}  \qquad (t,x)\in [0,T]\times\Gamma\label{BGam}\\ 
  \nabla\rho\cdot {\bf n(x)}&=0 \qquad (t,x)\in [0,T]\times\Gamma_{a}\label{BGamx}
\end{align}
where {\bf n(x)} denotes the outer normal vector to the boundary. For the Eikonal equation \eqref{EikEq}, we consider the non-homogenous Dirichlet boundary condition
\begin{equation}
  \phi = \phi_{b}  \qquad (t,x)\in [0,T]\times\partial\Omega \label{EikonalBC}
  \end{equation}
  where the function $\phi_{b}$ takes zeros on the exit $\Sigma$. The system \eqref{FluidEq}-\eqref{FluidIn} is a highly nonlinear coupled system of partial differential equations (PDE). Few analytic results are available, all of them restricted to spatial dimension one. The main difficulty comes from the low regularity of the potential $\phi$, which is only Lipschitz continuous. For existence and uniqueness results of a regularized problem in 1D, we refer the reader to \cite{burger2014mean,amadori2023mathematical} and references therein.

%
%
%
%
%
%
%
%
%
%
%
%
%
In this paper, we present a generalization of the Hughes model, which tries to avoid the high-density with local vision via partial knowledge of the pedestrian density. First, we discuss the proper modeling setup, the microscopic description, and the derivation of the macroscopic generalized Hughes (GH) model. The generalization of the model consists of taking on consideration human behavioral by avoiding or traveling to the high-density regions. 

Quantification of properties related to pedestrian crowds plays a pivotal role in comprehending pedestrian flows and implementing real-time control mechanisms. Often, crowd models rooted in physics and fluid dynamics represent individuals as particles, yielding satisfactory results under specific conditions. However, these models exhibit limitations in their general applicability, underscoring the importance of gaining a deeper understanding of the cognitive and psychological facets of human motion within crowds. In our present study, we propose the introduction of a new parameter, denoted as $\beta$, into the Hughes model, serving as a safety crowd parameter. The parameter $\beta$ is versatile and can be tailored to various characteristics of pedestrians, such as psychological profiles, walking speeds, perception abilities, traversal abilities, physical sizes, etc.

We establish the relationship between the GH model and the mean-field games (MFGs for short) system, providing an overview of recent developments in the field of mean field games. Additionally, we give proof of the existence of weak solutions for the GH model and, under appropriate monotonicity conditions, demonstrate uniqueness results. Subsequently, we study the vanishing viscosity limit of weak solutions. 


Finally, we present a set of numerical results aimed at dissecting the impact of the parameter $\beta$ within our intelligent fluid model. For these numerical experiments, we employ the finite element method (FEM), implemented through the FreeFem++ library. It is worth noting that this choice significantly simplifies the computational complexity for the authors; however, it necessitates the inclusion of some viscosity within the model to ensure numerical convergence. Additionally, alternative numerical approaches such as the Finite Volume Scheme (FVM), which has been successfully applied to the viscous-hughes model in prior work \cite{twarogowska2014macroscopic}, can offer valuable insights and comparisons. In fact, in a related study \cite{twarogowska2014macroscopic}, the authors conducted a comparative analysis of solution behaviors between the Hughes model and a second-order model based on a mixed finite volume method. Furthermore, we explore the application of the WENO scheme for the density equation and the sweeping WENO method for the HJB-type equation, as demonstrated in previous research by Huang et al. \cite{huang2009revisiting}. However, we have the finite difference method for MFGs introduced in \cite{achdou2010mean} for stationary and time-dependent MFGs. The Lax-Friedrichs or Godunov-type schemes are introduced to approximate Hamiltonians. The Kolmogorov-type equation is discretized in such a way that it preserves the adjoint structure in the MFG. We also refer the readers to \cite{achdou2016convergence, bakaryan2023discrete} and the references therein for numerical methods and convergence results of different discrete schemes for second-order mean field games.

%


This paper is structured as follows. We start in Section \ref{sec0} with a microscopic description of the GH model. Then, this microscopic setup leads in the mean-field limit to the GH model which is a class of MFGs systems. An overview of the results for the existence and uniqueness of the MFG system is presented in Section \ref{sec3}.  In Section \ref{sec4}, we give the existence and uniqueness of weak solution for our fluid model. We then analyze the vanishing viscosity limit of weak solutions. Finally, in Section \ref{sec6} we show some numerical experiments based on the FEM discussing the effect of the parameter $\beta$ and viscosity in the GH model.

\section{Generalized Hughes Model} \label{sec0}

\subsection{ Microscopic description: Mean-field limit} \label{sec1}
The objective of this section is to give a derivation of our GH model from the microscopic level and its mean-field limit. We consider a system of controlled individuals (also referred to as particles or players) in mutual interaction, each pedestrian (or player) controls his own private state $x(t)$ at time $ t \in [0,T]$, when $T>0$, driven by the vector-valued drift function $b$ value in some subset $\mathcal{A}$ of an euclidian space. We assume that $x(t)$ is a Markov diffusion process governed by the following stochastic differential equation (SDE):
\begin{equation}\label{EQStoc}
dx(t)=b(t,x(t),u(t,x(t))) dt  + \sqrt{2\sigma} dW(t)
\end{equation}
where $dW(t) \in \mathbb{R}^{2}$ is a $2d$ brownian motion process, with stochastically independent components. $x(t)$ denotes the position of a pedestrian in a domain $\Omega \subset \mathbb{R}^{2}$, and $b = (b_{1}, b_{2})$ denotes the velocity field of motion of the pedestrian. For simplicity,
we assume that the dispersion coefficient $\sigma$ is a constant but it can depend on the control strategy. In the control theory framework, the function $u$ uses the current value $x(t)$ to affect the dynamics of the stochastic process by adjusting the drift function. Similarly to \eqref{EQStoc}, we consider the following \textit{action functional}
\begin{equation}\label{EQFUNctional}
J\left(t,x,u \right)=\mathbb{E} \Big [ \int_{t}^{T} \mathbf{L}\left(x(s),u(s,x(s))\right) ds+g(x(T))\,\,\, | \,\,\, x(t)=x \Big ],
\end{equation}
which is an expectation to the process $x(t)$ taking the value $x$ at time $t$. We refer to the functions $\mathbf{L}$ and $g$ as the running cost and the terminal cost functions, respectively. Now, the optimal control $\bar{u}$ that minimizes $J\left(t,x,u \right)$ for the process \eqref{EQStoc} is given by
\begin{equation}\label{OptimalcontrolU}
\bar{u}=\text{argmin}_{u\in \mathcal{A}} J\left(t,x,u \right).
\end{equation}
Furthermore, we define the value function $\phi$ of the stochastic dynamics control problem as 
\begin{equation}\label{Valuefun}
\phi(t,x)=\min_{u\in \mathcal{A}} J\left(t,x,u \right)=J\left(t,x,\bar{u} \right).
\end{equation}
The value function is a critical element in this framework. For each player, it is the cost functional of the agent determined in the Nash equilibrium. Isaacs in \cite{isaacsdifferential} was the first to link the value function in these differential games to the Hamilton-Jacobi equations. Crandall, Ishii, and Lions then further developed this concept in their work on viscosity solutions; see \cite{ishii1990viscosity,crandall1992user} for a comprehensive overview.

We now state the theorem showing that the value function $\phi$ is the solution to the Hamilton-Jacobi-Bellman (HJB) equation.
\begin{thm}\label{TheoremHJBE}
Assume that $x(t)$, for $t\in[0,T]$ solving \eqref{EQStoc} with a control function $u$ and that the function $\phi$ defined by \eqref{Valuefun} is bounded and smooth. Then, the value function $\phi$ verifies the following HJB equation.
\begin{equation}\label{Hamiltoneq}
\begin{cases}
-\partial_{t}\phi-H(t,x,\nabla_{x}\phi,\nabla^{2}_{x}\phi)=0,\qquad (t,x)\in [0,T)\times  \Omega\\
\phi(t,x)=g(x)\qquad\qquad\qquad \qquad\qquad x\in  \partial\Omega
\end{cases}
\end{equation}
where the Hamiltonian function is defined by 
\begin{equation}\label{EQHAM}
H(t,x, \nabla_{x}\phi,\nabla^{2}_{x}\phi) := \min_{v \in \mathcal{A}} \left(  \sum_{i=1}^{2}b_{i}(t,x,u(t,x))\partial_{x_{i}}\phi+\sigma\Delta \phi(t,x) +\mathbf{L}(x,u(t,x))   \right) .
\end{equation}
\end{thm}
For more details  we refer the reader to \cite[Theorem 3.1  page 157]{zbMATH05032360}

By assuming differentiability with respect to the control function $u$ in \eqref{EQHAM}, then the optimal control $\bar{u}$ should satisfy at each time $t$ and for each point $x$ the optimality condition as follows: 
\begin{equation}\label{OptimalityHAM}
 \sum_{i=1}^{2}\partial_{u} b_{i}(t,x,u(t,x))\partial_{x_{i}}\phi  +\partial_{u}\mathbf{L}(x,u(t,x)) =0.
\end{equation}
Then by using the assumption that this process has an absolutely continuous probability measure, we can formulate the expectation in \eqref{EQFUNctional} in terms of probability density functions (PDF) governed by the Kolmogorov or Fokker-Planck (FP) equation problem with initial density distribution $\rho_{0}=\delta_{x(t)-x}$ at time $t=0$. Thus, the functional \eqref{EQFUNctional} becomes
\begin{equation}\label{EQFUNctionalcont}
J\left(u,\rho \right)= \int_{t}^{T}\int_{\Omega}\mathbf{L}\left(x,u(t,x)\right) dx ds +  \int_{\Omega}g(x)\rho(T,x) dx
\end{equation}
Therefore, we can state the optimization problem \eqref{OptimalcontrolU} as an FP optimal control problem where an optimal control $u$ in the admissible set $\mathcal{A}$ sought that minimizes the functional \eqref{EQFUNctionalcont}. We are identifying the chosen admissible set of Markov control policies to the admissible set of controls in the FP optimal control formulation. To characterize the optimal FP solution to this problem, we introduce the following Lagrange function.
\begin{equation}
\begin{aligned}
\mathrm{L}(\rho, \phi,u) :=&\int_{0}^{T}\int_{\Omega}\mathbf{L}\left(x,u(t,x)\right) dx ds +  \int_{\Omega}g(x)\rho(T,x) dx+\int_{0}^{T}\int_{\Omega}\phi(s,x)\Big[-\partial_{s}\rho(s,x) \\& - \sum_{i=1}^{2}\partial_{x_{i}}\left(b_{i}(t,x,u(t,x))\rho(t,x)\right)+\sigma\Delta\rho(t,x) \Big] dx ds 
\end{aligned}
\end{equation}

Consequently, the optimal control solution is characterized as the solution to the following optimality system
\begin{equation}\label{XZX0}
\partial_{\phi}\mathrm{L}(\rho, \phi,u)=0 	\iff  \quad \begin{cases}
-\partial_{t}\rho-\sum_{i=1}^{2}\partial_{x_{i}}\left(b_{i}(x,u(t,x))\right)\rho(t,x)+\sigma\Delta\rho(t,x)=0,\\
\rho(t=0,x)=\rho_{0}(x),
\end{cases}
\end{equation}

\begin{equation}\label{XZX1}
\partial_{\rho}\mathrm{L}(\rho, \phi,u)=0  	\iff \quad \begin{cases}
\partial_{t}\phi +\sum_{i=1}^{2}b_{i}(x,u(t,x))\partial_{x_{i}}\phi+\sigma\Delta\phi(t,x) +\mathbf{L}(x,u(t,x))=0,\\
\phi(T,x)=g(x),
\end{cases}
\end{equation}

and the optimality condition 
\begin{equation}\label{OptimaltyCond}
\partial_{u}\mathrm{L}(\rho, \phi,u)=0 	\iff \quad \rho(t,x)\left(  \sum_{i=1}^{2} \partial_{u}b_{i}(x,u(t,x)) \partial_{x_{i}}\phi  +\partial_{u}\mathbf{L}(x,u(t,x)) \right) =0.
\end{equation}

A sufficient condition to ensure that \eqref{OptimaltyCond} holds is  to satisfy the optimality condition \eqref{OptimalityHAM} within the context of the Hamilton-Jacobi-Bellman (HJB) formulation.

\subsection{ Macroscopic description of the GH model} \label{sec21}
This section is dedicated to the formal derivation of both the generalized and classical Hughes models within a macroscopic context, deriving them from the system defined by equations \eqref{XZX0}-\eqref{OptimaltyCond}. To begin, we will present the derivation of the Hughes model. We first take the following.

\[
b=u,\quad  \quad  \mathbf{L}(x,u)=\frac{u^{2}}{2f^{2}(\rho)} +\frac{1}{2},
\]
where $f(\rho)=(\rho_{max}-\rho)_{+}$. 

Consequently, from the optimality condition \eqref{OptimaltyCond}, we get a feedback in this form $u=-f^{2}(\rho)\nabla \phi$. Then, system \eqref{XZX0} -\eqref{OptimaltyCond} becomes
\begin{equation}\label{XZXx1}
\begin{cases}
\partial_{t}\rho -div(\rho f^{2}(\rho)\nabla \phi )-\sigma\Delta \rho=0\\
-\partial_{t}\phi +f^{2}(\rho)|\nabla\phi|^{2}-\sigma\Delta\phi-\frac{1}{2}f^{2}(\rho)|\nabla\phi|^{2} -\frac{1}{2}=0.
\end{cases}
\end{equation}
Our core assumption is that the pedestrian lacks the ability to foresee the future evolution of the population. Instead, they base their strategic decisions solely on real-time information, without the ability to predict forthcoming events. This strategic approach serves as a foundational principle supporting a particular class of quasi-stationary Mean Field Games (MFG) systems, as discussed in \cite{mouzouni2020quasi} and \cite{camilli2023quasi}.

In this context, individuals operate within a static environment. At each moment, they determine their course of action solely based on current information, devoid of any foresight into the future. Consequently, this leads to the  stationary HJB equation  coupled with an evolutive FP equation. Moreover, if we assume a complete absence of viscosity ($\sigma = 0$) within the model, we can deduce the classical Hughes model \eqref{FluidEq}-\eqref{EikonalBC}. In this macroscopic framework, individuals strive to minimize their travel time while conscientiously avoiding high-density regions.

However, we are now introducing a pioneering model that explicitly incorporates the element of congestion. To achieve this, we introduce a novel parameter, denoted as $\beta$, into the running cost term denoted by $\mathbf{L}$ (here, we adopt the notation $\mathcal{L}$ to reflect its dependence on the spatial variable $x$, density $\rho$, and the potential field $v$). This leads to the redefinition of the functional $\mathbf{L}$ as follows:

\begin{equation}\label{runningFunc}
\mathbf{L}(x,u)=\mathcal{L}(x,\rho,v)=\frac{1}{2} f^{-\beta}(\rho)\vert v\vert^{2} +\frac{1}{2} f^{2-\beta}(\rho).
\end{equation}
Therefore from the optimality condition \eqref{OptimaltyCond}, the optimal control situation  is given by $u=-f^{\beta}(\rho)\nabla \phi$. Then \eqref{XZX0} -\eqref{XZX1} becomes
\begin{equation}\label{sSa}
\begin{cases}
\partial_{t}\rho- div(\rho f^{\beta}(\rho)\nabla \phi )-\sigma\Delta \rho=0\\
-\partial_{t}\phi +\frac{1}{2} f^{\beta}(\rho)|\nabla\phi|^{2}-\sigma\Delta\phi -\frac{1}{2} f^{2-\beta}(\rho)=0.
\end{cases}
\end{equation}
Using the same argument that the generic pedestrian cannot predict the evolution of the population, but chooses its strategy only on the basis of the information available at the given instant of time, without anticipating, we can consider the quasi-stationary HJB equation as follows 
\begin{equation}\label{QuasiStat}
\begin{cases}
\partial_{t}\rho- div(\rho f^{\beta}(\rho)\nabla \phi )-\sigma\Delta \rho=0\\
\frac{1}{2} f^{\beta}(\rho)|\nabla\phi|^{2}-\sigma\Delta\phi -\frac{1}{2} f^{2-\beta}(\rho)=0.
\end{cases}
\end{equation}
Furthermore, assuming the absence of viscosity ($\sigma = 0$) in the model, we derive the following system
\begin{equation}\label{XZX1G}
\begin{cases}
\partial_{t}\rho -div(\rho f^{\beta}(\rho)\nabla \phi )=0\\
f^{2\beta -2}(\rho)|\nabla\phi|^{2} =1.
\end{cases}
\end{equation}
When $\beta = 2$, the Hughes model is obtained, in which people try to stay away from areas of high density. However, when $\beta =0$, the HJB equation suggests a connection between velocity and density, with $ |\nabla\phi| =f(\rho)$. This leads to a situation in which individuals are drawn to areas of high density, a phenomenon known as concentration. This is in stark contrast to the Hughes model, where people usually try to avoid high-density areas. An example of this behavior can be seen during the Tawaf ritual at the Kaaba, where Muslim pilgrims circle the Kaaba and many attempt to pause and kiss the Al Hajaru al Aswad.
When $\beta = 1$, the speed of pedestrians remains constant regardless of the density of the area they are in. This parameter is heavily influenced by the psychological state of the pedestrians, which is determined by the event or location. The value of $\beta$ reflects the adaptive walking strategies of pedestrians, depending on their psychological state in response to different environmental conditions or events. As an example, consider the Tawaf ritual at the Kaaba, which is performed during Hajj or Umrah. When a worshiper enters the mosque precinct. This leads to a defined area known as the \textit{Sahan}. Worshipers in this area can be divided into two groups: those who want to get close to the Kaaba and kiss the Black Stone (\textit{Al Hajar Aswad}), and those who want to avoid densely populated areas, often due to age or health issues. For the first group, their main goal is to reach the \textit{Al Hajar Aswad}, so they are not concerned with the crowd density. In this case, when $\beta = 1$, it implies that $ \nabla\phi = 1$. On the other hand, the second group of worshipers is more focused on navigating through the crowd density, as they want to avoid the hustle and bustle typically associated with densely populated areas around the Kaaba. This group includes people of advanced age, those with health concerns, expectant mothers, and others. In this case, $\beta$ remains consistently greater than 1, specifically within the range $\beta \in (1, 2]$. For $\beta \in [0, 1)$, there are at least three different scenarios to consider. To begin with, there is the case of \textit{curiosity-driven crowds}, such as those attending superstar concerts or political gatherings, who are drawn to high-density locations out of interest. Second, there is the scenario of \textit{panic}, where people flock to high-density regions because they believe that safety or something valuable can be found there. Lastly, there is the case of \textit{fear}, where individuals, when feeling scared, tend to seek proximity to high-density areas as a response. These three scenarios represent different motivations for individuals in situations with $\beta \in [0, 1)$.

%
%

We now assume that the curve is described parametrically in two-dimensional (2D) by the function ${\bf x}(s) = \left(x_{1}(s),x_{2}(s)\right)$. We define ${\bf z}(s)=\phi({\bf x}(s) )$, and 
\[
 {\bf p}(s)=\nabla_{{\bf x}}\phi({\bf x}(s) )= \left(p_{1}(s),p_{2}(s)\right), \quad \text{such that}  \quad p_{i}(s)=\phi_{x_{i}}({\bf x}(s) ), \qquad i\in\{1,2\},
\]
where  ${\bf z}(.)$ gives the values of $\phi$ along the curve ${\bf x}(.)$,  ${\bf p}(.)$ records the values of the gradient $\nabla_{{\bf x}}\phi$ and we define  ${\bf q}(s)=\rho({\bf x}(s) )$ . We set 
\[
\dot{ {\bf p}}(s)=D_{v}\mathcal{L}({\bf x}(s),{\bf q}(s),{\bf p}(s)),\qquad (0\le s\le t)
\]
where ${\bf p}(.)$ is called the generalized momentum corresponding to the position  ${\bf x}(.)$ and the velocity $\dot{ {\bf x}}$. Now we make the same assumption as \cite[(9), p.118]{evans2010partial}:
\begin{equation}\label{HamLag}
\begin{cases}
\text{Suppose  for all }\,\,\,\,\, x,p \in\mathbb{R}^{2}\,\,\, \text{ that the equation} \\
p=D_{v}\mathcal{L}(x,q,p),\\
\text{can be uniquely solved for}\,\, v \,\,  \text{as a smooth function of}\,\, p \,\, \text{and}\,\, x, v= {\bf v}(x,p).
\end{cases}
\end{equation}
The Hamiltonian $\mathcal{H}$ associated with the Lagrangian operator $v\longmapsto \mathcal{L}(x,\rho,v)$, is given by
\begin{equation}\label{EqHamil0}
\mathcal{H}(x,\rho,p)= p\cdot {\bf v} - \mathcal{L}(x,\rho,{\bf v}),
\end{equation}
where the function ${\bf v}(.)$ is defined implicitly by \eqref{HamLag}. Formally, we assume that $\mathcal{H}$ only depends on $p$, not on $v$.
\begin{equation}\label{EqHamil1}
D_{v}\mathcal{H}(x,\rho,p) = 0 = p- D_{v}\mathcal{L}(x,\rho,{\bf v})
\end{equation}
Then, we have ${ p =f^{-\beta}(\rho){\bf v}}$. Replacing the expression of ${\bf v}$ in  \eqref{EqHamil0}, we get 
\begin{equation}\label{EqHamil}
\begin{aligned}
\mathcal{H}(x,\rho,p)=& p\cdot {\bf v} -\left( \frac{1}{2} f^{-\beta}(\rho)\vert {\bf v}\vert^{2} +\frac{1}{2} f^{2-\beta}(\rho)\right)\\&
=f^{\beta}(\rho)|p|^{2} -\left( \frac{1}{2} f^{-\beta}(\rho) f^{2\beta}(\rho)\vert p\vert^{2} +\frac{1}{2} f^{2-\beta}(\rho)\right)\\&
=\frac{1}{2} f^{\beta}(\rho)|p|^{2} -\frac{1}{2} f^{2-\beta}(\rho),
\end{aligned}
\end{equation}
we denote by $\mathcal{H}$ (or $\mathcal{L}^{\ast}$) the Legendre transform defined by 
\begin{equation}\label{EqHamilX}
\mathcal{H}(x,\rho,p)=\mathcal{L}^{\ast}(x,\rho,p)=\sup_{q\in \mathbb{R}^{2}}[q\cdot p -\mathcal{L}(x,\rho,q)],
\end{equation}
and 
\begin{equation}\label{EqHamilY}
\mathcal{L}(x,\rho,{\bf v})=\mathcal{H}^{\ast}(x,\rho,{\bf v})=\mathcal{H}_{p}(x,\rho,p)\cdot p-\mathcal{H}(x,\rho,p),\quad \text{where}\,\,\,\,{\bf v}=\mathcal{H}_{p}(x,\rho,p),
\end{equation}

for more details we refer to  \cite[Chapter 3]{evans2010partial}. 
Now, we can write system \eqref{sSa} in the structure of second order MFGs system as follows
\begin{align}
\partial_t\rho - \nabla \cdot \left( \rho \mathcal{H}_{p}(x,\rho,\nabla \phi)\right) -\sigma \Delta  \rho= 0, &\qquad t>0, \quad x\in\Omega \label{EqFluidHd0}\\ 
-\partial_t\phi + \mathcal{H}(x,\rho,\nabla \phi) - \sigma \Delta  \phi=0, &\qquad t>0, \quad x\in\Omega \label{EqEikonalHd0}\\
\rho(t=0,x)= \rho_{0}(x),&\quad x\in\Omega \label{FluidIn2}\\ 
\phi(t=T,x)= \phi_{T}(x),&\quad x\in\Omega, \label{EikonalIn2}
\end{align}

where $\mathcal{H}$ is a convex function of $p$ and $\mathcal{H}_{p}$ represents $\frac{\partial\mathcal{H}}{\partial p}( x, \rho, p)$.
The following boundary conditions for the density equation are given by 
\begin{align}
\rho = 0,&\qquad t>0, \quad x\in\Sigma \label{Boundrho1}\\ 
\left(\sigma\nabla\rho +\rho\mathcal{H}_{p}(x,\rho,\nabla \phi)\right)\cdot n(x)= 0,&\qquad t>0, \quad x\in\Gamma_{a} \label{Boundrho2}\\ 
 \rho=\rho_{b}, &\qquad t>0, \quad x\in\Gamma, \label{Boundar3}
\end{align}
and for the HJB equation we consider the following boundary conditions 
\begin{align}
\phi = 0,&\qquad t>0, \quad x\in\Sigma \label{Boundphi1}\\ 
\nabla \phi\cdot n(x)= 0,&\qquad t>0, \quad x\in\Gamma_{a} \label{Boundphi2}\\ 
 \phi=\phi_{b}, &\qquad t>0, \quad x\in\Gamma. \label{Boundar3}
\end{align}


\subsection{Review of the literature in connection to the Mean-Field Games}\label{sec3}

 Previously, we have already shown that the GH model has the structure of a non-separable MFGs. In this section, we present an overview of the recent advanced results on the MFGs system and its derivation. MFGs have been introduced in the mathematics literature by Lasry and Lions as limits of problems from game theory, as the number of agents tends to infinity \cite{lasry2006jeux,lasry2006jeux1,lasry2007mean}. From a control theory perspective, mean field games were also introduced around the same time by Huang, Caines, and Malhame \cite{huang2007large,caines2006large}. This theory is  motivated by problems in economics and engineering and with the goal of approximating Nash equilibria of games with a large number of symmetric agents. Since their introduction, MFGs have been extensively studied in the literature and several research topics have been addressed, from both theoretical and applied perspectives. The main goal is typically to study the equilibria of such games, which are usually characterized as solutions of a system of PDEs. There are two classes of MFGs that are characterized by the separable and non-separable Hamiltonian $\mathcal{H}$. This means that it is assumed that there exist $H$ and $F$ such that $\mathcal{H}(t, x,\rho,p) = H(t,x,p)+ F(t,x,\rho)$. In such a case, the function $H$ is still known as the Hamiltonian, but $F$ is then referred to the coupling term. However, in applications, non-separable Hamiltonians are frequently of interest for example in economics, which does not tend to have this separable structure \cite{moll2016mean,achdou2014partial,achdou2022income} which is also the case of our intelligent fluid model \eqref{EqFluidHd0}-\eqref{EikonalIn2} and \eqref{Boundrho1}-\eqref{Boundar3}. 

In the case of the separable Hamiltonian, there are a number of works that prove the existence of solutions. The separability assumption, as well as other structural assumptions, such as the convexity of $H$ and the monotonicity of $F$, are helpful in the analysis of such models. For example the case of the second-order quadratic MFGs  with aggregation force 
\begin{align}
\partial_tm- \nabla \cdot \left( m\nabla u\right) - \Delta m= 0, &\qquad t>0, \quad x\in\mathbb{T}^{d} \label{EqFluidHd0C}\\ 
-\partial_t u + \frac{1}{2}|\nabla u|^{2}-  \Delta  u =-f(m) +V(x), &\qquad t>0, \quad x\in\mathbb{T}^{d} \label{EqEikonalHd0C}\\
m(t=0,x)= m_{0}(x),&\quad x\in\mathbb{T}^{d} \label{FluidInC}\\ 
u(t=T,x)= u_{T}(x),&\quad x\in\mathbb{T}^{d} \label{EikonalInC}
\end{align}
with $\mathbb{T}^{d}$ is standard flat torus (the dimension space $d$ ), $m$ is the population density, $u$ the potential, $m_{0}$ is a smooth probability density, $u_{T}$ a smooth final cost, $V$ is a bounded potential, and $f$ is chosen
 \[
 f(m)=\pm\gamma m^{\alpha},
 \]
 where $\gamma$ and $\alpha$ are  related to the aggregation force, for more details see \cite{cirant2022existence} and references therein. Moreover, in \cite{porretta2014planning,porretta2015weak} Porretta proved the existence of weak solutions of the second-order MFGs system. Then, recently in \cite{griffin2022variational} Griffin-Pickering and Mészáros studied weak solutions to a first-order mean-field games system involving kinetic transport operators. Gomes and Pimentel, in \cite{gomes2015time}, showed the existence of a strong solution with logarithmic coupling.  Gomes, Pimentel, and Sanchez-Morgado proved in \cite{gomes2015time0,gomes2016time} the existence of a strong solution for the case of super-quadratic and subquadratic Hamiltonians. However, for more references on the existence of solutions for MFGs, we refer to \cite{ferreira2018existence,ferreira2019existence,cirant2021maximal} and the references therein.

To our knowledge,  the first work on the existence of strong solution to the non-separable Hamiltonian is proved by Ambrose in \cite{ambrose2018strong} . The author presented an existence for strong solutions of the non-separable MFG system. First, a smallness condition on the data is used, and a small parameter in front of the Hamiltonian is used to compensate for the nonlinear term.  We also found  the work of  Cirant, Gianni, and Mannucci  \cite{cirant2020short} on  MFGs with non-separable Hamiltonians. The authors proved an existence theorem for non-separable mean field games in Sobolev spaces, under a smallness condition on the time horizon.

In \cite{Lions2023}, P.-L.Lions gives the general structural conditions yielding the uniqueness for the nonseparable MFG systems with local coupling
\begin{align}
\partial_t m - \sigma\Delta m -div(m\mathcal{H}_{p}(x,m, \nabla u ))= 0, &\qquad t>0, \quad x\in\Omega \label{EqFluidHd0Cun}\\ 
-\partial_t u - \sigma\Delta  u +\mathcal{H}(x,m, \nabla u ) = F(t,x,m), &\qquad t>0, \quad x\in\Omega \label{EqEikonalHd0Cun}
\end{align}
namely, that $F$ and $G$ will increase w.r.t. $m$ and that the Lasry–Lions monotonicity condition 
\begin{equation}\label{uniq}
\mathbf{H}=\begin{pmatrix}
-\frac{2}{m}\mathcal{H}_{m}(t, x,m, p)& \mathcal{H}_{m,p}(t, x,m, p )\\
 \mathcal{H}_{m,p}(t, x,m, p)& 2\mathcal{H}_{p,p}(t, x,m, p )\\
\end{pmatrix} \ge 0,
\end{equation}
is satisfied, where $\mathcal{H}_{m}$ stands for $\frac{\partial\mathcal{H}}{\partial m}(x, m, p)$ , where $\mathcal{H}_{m,p}$ stands for $\frac{\partial\mathcal{H}}{\partial m \partial p}( x, \rho, p)$, for all $x\in \Omega$, $m>0$ and $p \in \mathbb{R}^{2}$. Recently, the system extended MFGs (EMFGs) was introduced by Lions and Souganidis in \cite{lions2020extended} who coined the term extended MFGs, to simultaneously study several MFGs type problems for which, in contrast to the case of standard MFG, the vector field $\mathcal{B}$ does not necessarily equal $\rho\mathcal{H}_{p}(x,m ,\nabla \phi )$
\begin{align}
\partial_t m  - \nabla \cdot \left(\mathcal{B}(x,m ,\nabla u )\right) = 0, &\qquad t>0, \quad x\in\Omega \label{EqFluidHd0xx}\\ 
-\partial_tu + \mathcal{H}(x,m ,\nabla u )=0, &\qquad t>0, \quad x\in\Omega \label{EqEikonalHd0xx}
\end{align}

%
%

where it was shown that (EMFGs) has at most one classical solution under some  sufficient condition. It was demonstrated in \cite{munoz2022classical,munoz2023classical} that classical solutions to EMFGs in any dimension exist when the initial density has a lower bound and the following blow-up assumption is satisfied
\begin{equation}\label{Eqcor}
\lim_{m \to 0^{+}}\mathcal{H}(x,m,\nabla u) = +\infty.
\end{equation}
Recently, \cite{porretta2023regularizing} proved a similar regularity result for the separable Hamiltonian $\mathcal{H}(x,m,p) = H(x,p)-f(m)$ with a lower bound on the terminal density $m_{T}$.

We should also mention that Lasry and Lions \cite{Lions2023} introduced the \textit{master equation} in order to reformulate the original MFG system in order to take into consideration a much more complex stochastic situation; the authors introduce the \textit{master equation} in order to encompass complex situations. However, the master equation can be understood as a nonlinear transport equation in the space of probability measures. The global well-posedness of master equations requires the uniqueness of mean field equilibrium, typically under certain monotonicity conditions, we refer to \cite{cardaliaguet2019master} for more details.





\section{A weak solution for  generalized Hughes Model }\label{sec4}
This section is devoted to the proof of the existence and uniqueness of weak solutions for system \eqref{EqFluidHd0}-\eqref{EikonalIn2}. For simplicity we assume that the equation takes place in a standard flat torus $\Omega= \mathbb{T}^{2}$. The function $\mathcal{H}(x, \rho, p)$ is assumed to be measurable with respect to $(t, x)$, continuous with respect to $\rho$ and $C^{1}$with respect to p. 
We recall  our  fluid model \eqref{EqFluidHd0}-\eqref{EikonalIn2}  in the flat torus $\Omega$
\begin{align}
\partial_t\rho^{\sigma} - \nabla \cdot \left( \rho^{\sigma} \mathcal{H}_{p}(x,\rho^{\sigma},\nabla \phi^{\sigma})\right) -\sigma \Delta  \rho^{\sigma}= 0, &\qquad (t,x)\in\mathrm{Q}_{T}= (0,T)\times\Omega \label{EqFluidHd0x}\\ 
-\partial_t\phi^{\sigma} + \mathcal{H}(x,\rho^{\sigma},\nabla \phi^{\sigma}) - \sigma \Delta  \phi^{\sigma}=0, &\qquad  (t,x)\in\mathrm{Q}_{T} \label{EqEikonalHd0x}\\
\rho^{\sigma}(t=0,x)= \rho_{0}(x),&\quad x\in\Omega \label{FluidInx}\\ 
\phi^{\sigma}(t=T,x)= \phi_{T}(x),&\quad x\in\Omega.\label{EikonalInx}
\end{align}
Here, we denote the initial data as $\rho_{0}$ and $\phi_{T}$, while the Hamiltonian $\mathcal{H}$ is represented by \eqref{EqHamil}. To illustrate the dependence on the parameter $\sigma$, we utilize the index ${\sf ^{\sigma}}$ in the context of $\left(\rho^{\sigma},\phi^{\sigma}\right)$.

Now we start by giving some assumptions on the initial data. 
\begin{equation*}
\text{{\bf (AS)}} \begin{cases}
 &\text{Let}\,\,\,\,  \rho_{0}\in C(\Omega), \,\,\rho_{0} \in (0,\rho_{m}).\\ 
 &\text{Let}\,\,\,\,\phi_{T} \in  C(\Omega),\,\,\text{and there exists}\,\,\, c_{0} \in \mathbb{R}\,\,\,\text{such that} \,\, \phi_{T}\ge c_{0}.  
\end{cases}
\end{equation*}
We will start by defining a weak solution for the system of equations \eqref{EqFluidHd0x}-\eqref{EikonalInx}.
\begin{e-definition}\label{defweaksig}
A pair $(\rho^{\sigma},\phi^{\sigma}) \in L^{1}(\mathrm{Q}_{T}) \times  L^{1}(\mathrm{Q}_{T})$ is a weak solution to \eqref{EqFluidHd0x}-\eqref{EikonalInx} if 
\begin{description}
\item[ (i)] $\rho^{\sigma} \in C(0,T;L^{1}(\Omega))$, $\rho^{\sigma}\ge 0$, $\phi^{\sigma} \in L^{\infty}(0,T;L^{1}(\Omega)) \in L^{\infty}(0,T;L^{1}(\Omega)) \cap L^{q}(0,T;W^{1,q}(\Omega))$ for every $q<\frac{4}{3}$.
\item[ (ii)]  $\rho^{\sigma} f^{\beta}(\rho^{\sigma}) |\nabla \phi^{\sigma}|^{2} \in L^{1}(\mathrm{Q}_{T})$ and  $f^{\beta}(\rho^{\sigma}) |\nabla \phi|^{2} \in L^{1}(\mathrm{Q}_{T})$
\item[ (iii)]  $\phi^{\sigma} \in L^{\infty}(0,T;L^{1}(\Omega))$ is a solution of the HJB type equation in the sense of distributions
\begin{equation}\label{weakphi}
\int_{0}^{T}\int_{\Omega} \phi^{\sigma} \partial_{t}\varphi dx dt  - \sigma \int_{0}^{T}\int_{\Omega} \phi^{\sigma} \Delta \varphi dx dt  + \int_{0}^{T}\int_{\Omega} \mathcal{H}(x, \rho^{\sigma},\nabla \phi^{\sigma}) \varphi dx dt = \int_{\Omega} u_{T}\varphi(T) dx,
\end{equation}
 for every $\varphi \in C_{c}^{\infty}((0,T]\times\Omega)$.
\item[ (iv)]  $\rho^{\sigma} \in L^{\infty}(0,T;L^{1}(\Omega)) \cap  L^{2}(0,T;H^{1}(\Omega))$ is a solution of the Kolmogorov equation in the sense of distributions
\begin{equation}\label{weakrho}
\int_{0}^{T}\int_{\Omega} \rho^{\sigma} \left( -\partial_{t}\psi  - \sigma  \Delta \psi  +  \mathcal{H}_{p}(x, \rho^{\sigma},\nabla \phi^{\sigma}) \cdot \nabla \psi \right) dx dt   = \int_{\Omega} \rho_{0}\psi(0) dx,
\end{equation}
 for every $\psi \in C_{c}^{\infty}([0,T)\times\Omega)$.
\end{description}
\end{e-definition}
A solution $\phi^{\sigma}$ of equation \eqref{weakphi} belongs to $C^{0}(0,T;L^{1}(\Omega))$. However, we also need to work with the subsolutions of the same equation, for which this kind of continuity may not hold. We recall a result  in \cite{cardaliaguet2015second,achdou2018mean} showing in the weak sense the continuity for sub-solutions.

\begin{lemma}\label{lemma33}
 Let $\phi \in L^{\infty}(0,T;L^{1}(\Omega))$ satisfy for some function $h\in L^{1}(\mathrm{Q}_{T})$ and $g \in L^{1}(\Omega)$
 \[ 
 \int_{0}^{T}\int_{\Omega}\phi \partial_{t}\varphi dx dt - \sigma  \int_{0}^{T}\int_{\Omega}\phi \Delta \varphi dx dt  \le  \int_{0}^{T}\int_{\Omega}h\varphi dx dt + \int_{\Omega}g(x)\varphi(T,x) dx   
 \]
  for every nonnegative function $\varphi \in C_{c}^{\infty}((0,T]\times\Omega)$.
Then, for any Lipschitz continuous map $\xi : \Omega \to \mathbb{R}$, the map $t \longmapsto \int_{\Omega} \xi(x)\phi(t,x)dx$ has BV representative on [0,T]. Moreover, if we note its right limit at $t\in [0, T)$, then the map $\xi \longmapsto \int_{\Omega} \xi(x)\phi(t^{+}, x) dx$ can be extended to a linear bounded form on $C(\Omega)$.
\end{lemma}

As a consequence of Lemma \ref{lemma33}, for any sub-solution $\phi^{\sigma}$ of \eqref{weakphi} we can define $\phi^{\sigma}(0^{+})$ as a bounded Radon measure on $\Omega$. For simplicity, we note $\phi^{\sigma}(0) =\phi^{\sigma}(0^{+})$.

\subsection{Approximation of \eqref{EqFluidHd0x}-\eqref{EikonalInx}}
\begin{align}
\partial_t\rho^{\sigma}_{\epsilon}  - \nabla \cdot \left( \rho^{\sigma}_{\epsilon} \mathcal{H}_{p}(x,\mathcal{T}_{\epsilon}\rho^{\sigma}_{\epsilon},\nabla \phi^{\sigma}_{\epsilon})\right) -\sigma \Delta  \rho^{\sigma}_{\epsilon} = 0, &\qquad (t,x)\in\mathrm{Q}_{T}  \label{EqFluidHd0xa}\\ 
-\partial_t\phi^{\sigma}_{\epsilon}  + \mathcal{H}(x, \mathcal{T}_{\epsilon}\rho^{\sigma}_{\epsilon},\nabla \phi^{\sigma}_{\epsilon} ) - \sigma \Delta  \phi^{\sigma}_{\epsilon} =0, &\qquad (t,x)\in\mathrm{Q}_{T}  \label{EqEikonalHd0xa}\\
\rho^{\sigma}_{\epsilon}(t=0,x)=  \rho_{\epsilon 0}(x),&\quad x\in\Omega  \label{FluidInxa}\\ 
\phi^{\sigma}_{\epsilon}(t=T,x)= \phi_{\epsilon T}(x),&\quad x\in\Omega \label{EikonalInxa}
\end{align}
where $\mathcal{T}_{\epsilon}\rho=\mathcal{T}_{\epsilon}(\rho)= \min(\rho,\frac{\rho_{m}}{1+\epsilon})$, $ \rho_{\epsilon 0}=m^{\epsilon}\ast \rho_{0}$. Here, $\ast$ denotes the convolution in the spatial variable and $m^{\epsilon}$ is a standard symmetric mollifier, i.e., $m^{\epsilon}(x)= \frac{1}{\epsilon^{2}} m(\frac{x}{\epsilon})$ for a non-negative function $m\in C_{c}^{\infty}(\mathbb{R}^{2})$ such that $\int_{\mathbb{R}^{2}} m(x) dx=1$.
\begin{lemma}\label{Lemmaex}
Let $(\rho^{\sigma}_{\epsilon},\phi^{\sigma}_{\epsilon})$ be a solution to the system \eqref{EqFluidHd0xa}-\eqref{EikonalInxa}. Then, we have 
 \begin{equation}\label{EstimateA0}
 \phi^{\sigma}_{\epsilon} \ge c_{0},
 \end{equation}
 \begin{equation}\label{EstimateA00}
\phi^{\sigma}_{\epsilon} \in L^{\infty}(0,T;L^{1}(\Omega)),\quad  \rho^{\sigma}_{\epsilon} \in L^{2}(0,T;H^{1}(\Omega)),
 \end{equation}
 
  \begin{equation}\label{EstimateA001}
\mathcal{H}(x,\mathcal{T}_{\epsilon}\rho^{\sigma}_{\epsilon},\nabla \phi^{\sigma}_{\epsilon})  \in L^{1}(\mathrm{Q}_{T}),
 \end{equation}

 and 
 \begin{equation}\label{EstimateA}
\begin{aligned}
 \int_{\Omega}\phi_{\epsilon T}\rho^{\sigma}_{\epsilon}(T)dx &+\int_{0}^{T}\int_{\Omega} \rho^{\sigma}_{\epsilon}\left(\mathcal{H}_{p}(x,\mathcal{T}_{\epsilon}\rho^{\sigma}_{\epsilon},\nabla \phi^{\sigma}_{\epsilon})\cdot \nabla \phi^{\sigma}_{\epsilon} -\mathcal{H}(x,\mathcal{T}_{\epsilon}\rho^{\sigma}_{\epsilon},\nabla \phi^{\sigma}_{\epsilon}) \right) dx dt \\&+ \frac{1}{2}  \int_{0}^{T}\int_{\Omega}f^{\beta}(\rho^{\sigma}_{\epsilon})|\nabla \phi^{\sigma}_{\epsilon}|^{2} dx dt  \le C,
\end{aligned}
 \end{equation}
for a positive constant $C$ depending on $\phi_{T}$, $\rho_{m}$, $\beta$, $\rho_{0}$ and $\sigma$.
\end{lemma}
\begin{proof}
We multiply the equation of $\phi^{\sigma}_{\epsilon}$ by $\rho^{\sigma}_{\epsilon}$ and equation of $\rho^{\sigma}_{\epsilon}$, by $\phi^{\sigma}_{\epsilon}$, integrate by parts some terms and subtract; we get
\[
\int_{\Omega}\phi^{\sigma}_{\epsilon}(0) \rho_{\epsilon 0} dx - \int_{\Omega}\phi_{\epsilon T}\rho^{\sigma}_{\epsilon}(T)dx - \int_{0}^{T}\int_{\Omega} \rho^{\sigma}_{\epsilon}\left(\mathcal{H}_{p}(x,\mathcal{T}_{\epsilon}\rho^{\sigma}_{\epsilon},\nabla \phi^{\sigma}_{\epsilon})\cdot \nabla \phi^{\sigma}_{\epsilon} -\mathcal{H}(x,\mathcal{T}_{\epsilon}\rho^{\sigma}_{\epsilon},\nabla \phi^{\sigma}_{\epsilon}) \right) dx dt =0,
\]
then
\[
\int_{\Omega}\phi^{\sigma}_{\epsilon}(0) \rho_{\epsilon 0} dx = \int_{\Omega}\phi_{\epsilon T}\rho^{\sigma}_{\epsilon}(T)dx +\frac{1}{2}  \int_{0}^{T}\int_{\Omega} \rho^{\sigma}_{\epsilon} f^{\beta}(\rho^{\sigma}_{\epsilon})|\nabla \phi^{\sigma}_{\epsilon}|^{2} dx dt  + \frac{1}{2}  \int_{0}^{T}\int_{\Omega}  \rho^{\sigma}_{\epsilon} f^{2-\beta}(\rho^{\sigma}_{\epsilon}) dx dt 
\]
which implies that 
\begin{equation}\label{Eqestimate0}
\begin{aligned}
 \int_{\Omega}\big[\phi^{\sigma}_{\epsilon}(0)\big]^{-}  \rho_{\epsilon 0} dx +\int_{\Omega}\phi_{\epsilon T}\rho^{\sigma}_{\epsilon}(T)dx +\frac{1}{2}  \int_{0}^{T}\int_{\Omega} \rho^{\sigma}_{\epsilon} f^{\beta}(\rho^{\sigma}_{\epsilon})|\nabla \phi^{\sigma}_{\epsilon}|^{2} dx dt  &+ \frac{1}{2}  \int_{0}^{T}\int_{\Omega}  \rho^{\sigma}_{\epsilon} f^{2-\beta}(\rho^{\sigma}_{\epsilon}) dx dt \\& \le C \|\rho_{0}\|_{L^{\infty}(\Omega)}\int_{\Omega}\big[\phi^{\sigma}_{\epsilon}(0)\big]^{+}  dx .
 \end{aligned}
 \end{equation}
 From the definition of the Hamiltonian $\mathcal{H}$ and and $\phi_{T}\ge c_{0}$, we can deduce by comparison that $\phi^{\sigma}_{\epsilon}(t) \ge c_{0} + c_{0}(T-t)$ and then there exists an absolute constant where
 \begin{equation}\label{labboubdn}
\big[\phi^{\sigma}_{\epsilon}(0)\big]^{-}\le C_{1}.
 \end{equation}
We integrate  the equation of $\phi^{\sigma}_{\epsilon}$, and  using \eqref{labboubdn}, we get
\begin{equation}\label{eqX0}
\int_{\Omega}\int_{\Omega}\big[\phi^{\sigma}_{\epsilon}(0)\big]^{+}  dx +\frac{1}{2}  \int_{0}^{T}\int_{\Omega}f^{\beta}(\rho^{\sigma}_{\epsilon})|\nabla \phi^{\sigma}_{\epsilon}|^{2} dx dt \le \int_{\Omega} \phi_{\epsilon T} dx + \frac{1}{2}  \int_{0}^{T}\int_{\Omega} f^{2-\beta}(\rho^{\sigma}_{\epsilon}) dx dt  +C_{1}.
\end{equation}
Therefore, we have 
\begin{equation}\label{eqX0ZA}
\int_{\Omega}\int_{\Omega}\big[\phi^{\sigma}_{\epsilon}(0)\big]^{+}  dx +\frac{1}{2}  \int_{0}^{T}\int_{\Omega}f^{\beta}(\rho^{\sigma}_{\epsilon})|\nabla \phi^{\sigma}_{\epsilon}|^{2} dx dt \le C(\rho_{m},\phi_{T},\Omega,T)+C_{1}.
\end{equation}
Combining  \eqref{eqX0ZA} with \eqref{Eqestimate0} we get
\begin{equation}\label{Eqestimate}
\begin{aligned}
 \int_{\Omega}\phi_{\epsilon T}\rho^{\sigma}_{\epsilon}(T)dx &+\frac{1}{2}  \int_{0}^{T}\int_{\Omega} \rho^{\sigma}_{\epsilon} f^{\beta}(\rho^{\sigma}_{\epsilon})|\nabla \phi^{\sigma}_{\epsilon}|^{2} dx dt  + \frac{1}{2}  \int_{0}^{T}\int_{\Omega}  \rho^{\sigma}_{\epsilon} f^{2-\beta}(\rho^{\sigma}_{\epsilon}) dx dt \\&+ \frac{1}{2}  \int_{0}^{T}\int_{\Omega}f^{\beta}(\rho^{\sigma}_{\epsilon})|\nabla \phi^{\sigma}_{\epsilon}|^{2} dx dt  \le C(\rho_{m}, \|\rho_{0}\|_{L^{\infty}}, \|\phi_{T}\|_{L^{\infty}},\Omega,T,C_{1}),
\end{aligned}
 \end{equation}
which can be written as follows
\begin{equation}\label{EqestimateX}
\begin{aligned}
 \int_{\Omega}\phi_{\epsilon T}\rho^{\sigma}_{\epsilon}(T)dx &+\int_{0}^{T}\int_{\Omega} \rho^{\sigma}_{\epsilon}\left(\mathcal{H}_{p}(x,\mathcal{T}_{\epsilon}\rho^{\sigma}_{\epsilon},\nabla \phi^{\sigma}_{\epsilon})\cdot \nabla \phi^{\sigma}_{\epsilon} -\mathcal{H}(x,\mathcal{T}_{\epsilon}\rho^{\sigma}_{\epsilon},\nabla \phi^{\sigma}_{\epsilon}) \right) dx dt \\&+ \frac{1}{2}  \int_{0}^{T}\int_{\Omega}f^{\beta}(\rho^{\sigma}_{\epsilon})|\nabla \phi^{\sigma}_{\epsilon}|^{2} dx dt  \le C(\rho_{m}, \|\rho_{0}\|_{L^{\infty}}, \|\phi_{T}\|_{L^{\infty}},\Omega,T,C_{1}).
\end{aligned}
 \end{equation}

 We integrate   $\phi^{\sigma}_{\epsilon}$ equation  in $(t,T)\times \Omega$, we get that 
\[
\int_{\Omega} \phi^{\sigma}_{\epsilon}(t,x) dx +\frac{1}{2}\int_{t}^{T}\int_{\Omega} \mathcal{H}(x,\mathcal{T}_{\epsilon}\rho^{\sigma}_{\epsilon},\nabla \phi^{\sigma}_{\epsilon}) dx dt = \int_{\Omega} \phi_{\epsilon T} dx.
\]
Consequently, from \eqref{labboubdn}-\eqref{eqX0ZA} and the definition of $f$ we deduce  that 
\[
\mathcal{H}(x,\mathcal{T}_{\epsilon}\rho^{\sigma}_{\epsilon},\nabla \phi^{\sigma}_{\epsilon})  \in L^{1}(\mathrm{Q}_{T}),
\] 
and then we have 
\[
\int_{\Omega} \phi^{\sigma}_{\epsilon}(t,x) dx \le C. 
\]
Finally, we have already proved that $\phi^{\sigma}_{\epsilon}$ has a lower-bound, we deduce that $\phi^{\sigma}_{\epsilon} \in L^{\infty}(0,T;L^{1}(\Omega))$.

We multiply the equation \eqref{EqFluidHd0xa} with $(\rho_{m}-\rho^{\sigma}_{\epsilon})$ and we integrate over space and time, we get 

\[
\int_{0}^{T}\int_{\Omega}\partial_t\rho^{\sigma}_{\epsilon}(\rho_{m}-\rho^{\sigma}_{\epsilon}) - \int_{0}^{T}\int_{\Omega}\nabla \cdot \left( \rho^{\sigma}_{\epsilon} \mathcal{H}_{p}(x,\mathcal{T}_{\epsilon}\rho^{\sigma}_{\epsilon},\nabla \phi^{\sigma}_{\epsilon})\right)(\rho_{m}-\rho^{\sigma}_{\epsilon}) dx dt  -\sigma \int_{0}^{T}\int_{\Omega} \Delta  \rho^{\sigma}_{\epsilon} (\rho_{m}-\rho^{\sigma}_{\epsilon}) dx dt = 0,
\]
it yields that 
\begin{equation}
\begin{aligned}
-\frac{1}{2}\int_{0}^{T}\int_{\Omega}\partial_t(\rho_{m}-\rho^{\sigma}_{\epsilon})^{2} dx dt &-\int_{0}^{T}\int_{\Omega} \rho^{\sigma}_{\epsilon} \mathcal{H}_{p}\left(x,\mathcal{T}_{\epsilon}\rho^{\sigma}_{\epsilon},\nabla \phi^{\sigma}_{\epsilon}\right)\cdot \nabla \rho^{\sigma}_{\epsilon}dx dt  \\&-\sigma \int_{0}^{T}\int_{\Omega} |\nabla  \rho^{\sigma}_{\epsilon}|^{2}  dx dt = 0,
\end{aligned}
 \end{equation}
 then
   \begin{equation}
\begin{aligned}
\frac{1}{2}\|(\rho_{m}-\rho^{\sigma}_{\epsilon})(t)\|^{2}_{L^{2}}  &+ \int_{0}^{T}\int_{\Omega} \rho^{\sigma}_{\epsilon} (\rho_{m}-\mathcal{T}_{\epsilon}\rho^{\sigma}_{\epsilon})^{\beta}\nabla \phi^{\sigma}_{\epsilon}\cdot \nabla \rho^{\sigma}_{\epsilon}dx dt  \\&+\sigma \int_{0}^{T}\int_{\Omega} |\nabla  \rho^{\sigma}_{\epsilon}|^{2}  dx dt = \frac{1}{2}\|(\rho_{m}-\rho^{\sigma}_{\epsilon})(0)\|^{2}_{L^{2}}.
\end{aligned}
 \end{equation}
By Young inequality and \eqref{eqX0ZA} it yields that 
 \begin{equation}
\begin{aligned}
\frac{1}{2}\|(\rho_{m}-\rho^{\sigma}_{\epsilon})(t)\|^{2}_{L^{2}} &+\sigma \int_{0}^{T}\int_{\Omega} |\nabla  \rho^{\sigma}_{\epsilon}|^{2} dx dt  \le C(\rho_{m}, \|\rho_{0}\|_{L^{\infty}}, \|\phi_{T}\|_{L^{\infty}},\Omega,T,C_{1}) +\frac{1}{2}\|(\rho_{m}-\rho^{\sigma}_{\epsilon})(0)\|^{2}_{L^{2}} 
\end{aligned}
 \end{equation}
 which gives that $\rho^{\sigma}_{\epsilon} \in L^{2}(0,T;H^{1}(\Omega))$ and  $\rho_{m}-\rho^{\sigma}_{\epsilon} \in L^{\infty}(0,T;L^{2}(\Omega))$  which completes the proof. 
\end{proof}
Now, inspired by \cite{achdou2018mean,porretta2015weak}, we show some compactness properties for the family $\left( \rho^{\sigma}_{\epsilon}, \phi^{\sigma}_{\epsilon}\right)$ solution to the system \eqref{EqFluidHd0xa}-\eqref{EikonalInxa}.

\begin{thm}\label{Thmlim}
Under the assumptions \text{{\bf (AS)}}, it follows that 

\begin{description}
\item[(1) ]  There exist $\rho^{\sigma} \in C(0,T;L^{1}(\Omega))$ and $\phi^{\sigma} \in L^{\infty}(0,T;L^{1}(\Omega)) \cap L^{q}(0,T;W^{1,q}(\Omega))$ for every $q<\frac{4}{3}$ such that we can extract $\rho^{\sigma}_{\epsilon} \to \rho$ in $L^{1}(\mathrm{Q}_{T}))$. Moreover, we have  
\[
\phi^{\sigma}_{\epsilon} \to \phi^{\sigma}\,\,\, \text{and} \,\,\, \nabla \phi^{\sigma}_{\epsilon} \to \nabla \phi^{\sigma}, \quad \text{ almost everywhere.}
\]
Furthermore, $\phi^{\sigma}$ and $\rho^{\sigma}$ satisfy 

 \begin{equation}\label{EstimateAlim}
 \begin{aligned}
 \int_{0}^{T}\int_{\Omega} \rho^{\sigma} f^{\beta}(\rho^{\sigma}_{\epsilon})|\nabla \phi^{\sigma}|^{2} dx dt  + & \int_{0}^{T}\int_{\Omega}  \rho^{\sigma} f^{2-\beta}(\rho^{\sigma}) dx dt\\&+ \frac{1}{2}  \int_{0}^{T}\int_{\Omega}f^{\beta}(\rho^{\sigma})|\nabla \phi^{\sigma}|^{2} dx dt  \le C,
 \end{aligned}
 \end{equation}
for a positive constant C depending on $\phi_{T}$,$\rho_{m}$, $\beta$ and $\rho_{0}$.
\item [(2) ] We have $\rho^{\sigma}_{\epsilon}  \to \rho^{\sigma}$  in $C(0,T;L^{1}(\Omega))$ for every $t\in [0,T]$ and   \eqref{weakrho} holds  for every $\psi \in C_{c}^{\infty}([0,T)\times\Omega)$.
\end{description}
\end{thm}

\begin{proof}
By Lemma \ref{Lemmaex}, we have $\mathcal{H}(x,\mathcal{T}_{\epsilon}\rho^{\sigma}_{\epsilon},\nabla \phi^{\sigma}_{\epsilon})$ is bounded in $L^{1}(\mathrm{Q}_{T})$ uniformly in $\epsilon$. Therefore, $\left(-\partial_{t}\phi^{\sigma}_{\epsilon}-\sigma\Delta\phi^{\sigma}_{\epsilon}\right)_{\epsilon}$ is bounded in $L^{1}(\mathrm{Q}_{T})$. By the classical results on the parabolic equation with $L^{1}$ data it follows the compactness of $\phi^{\sigma}_{\epsilon}$ and ${\nabla \phi^{\sigma}_{\epsilon}}$ in $L^{1}(\mathrm{Q}_{T})$, see \cite[ Thoerem 2.1 and Theorem 3.3]{boccardo1997nonlinear}. Then, there exists $\phi^{\sigma} \in L^{\infty}(0,T;L^{1}(\Omega)) \cap L^{q}(0,T;W^{1,q}(\Omega))$ for every $q<\frac{4}{3}$ such we can extract a subsequence such that $\phi^{\sigma}_{\epsilon} \to \phi^{\sigma}$ and $\nabla\phi^{\sigma}_{\epsilon} \to \nabla \phi^{\sigma}$ in $L^{1}(\mathrm{Q}_{T})$ almost everywhere. 

Now for the density $\rho^{\sigma}_{\epsilon}$, from  Theorem \ref{Lemmaex} \eqref{EstimateA00}, we have $\rho^{\sigma}_{\epsilon} \in L^{2}(0,T;H^{1}(\Omega))$. Then ,we can apply the compactness results in  \cite[Theorem 6.1]{porretta2015weak}. We can extract a subsequence converges in  $L^{1}(\mathrm{Q}_{T})$ such that $\rho^{\sigma}_{\epsilon} \to \rho^{\sigma}$. Moreover, from the estimate \eqref{EstimateA}, and Fatou's Lemma, we get the estimate \eqref{EstimateAlim}. Now, we prove that $\rho^{\sigma}_{\epsilon}\mathcal{H}_{p}(x,\mathcal{T}_{\epsilon}\rho^{\sigma}_{\epsilon},\nabla \phi^{\sigma}_{\epsilon})$ strongly converges in $L^{1}(\mathrm{Q}_{T})$. Indeed, we have

 \begin{equation}
 \begin{aligned}
\int_{0}^{T}\int_{\Omega} | \rho^{\sigma}_{\epsilon}\mathcal{H}_{p}(x,\mathcal{T}_{\epsilon}\rho^{\sigma}_{\epsilon},\nabla \phi^{\sigma}_{\epsilon}) | dx dt  \le & \int_{0}^{T}\int_{\Omega} \rho^{\sigma}_{\epsilon} f^{\beta}(\rho^{\sigma}_{\epsilon})|\nabla \phi^{\sigma}_{\epsilon} | dx dt \\&  \le  \left(\int_{0}^{T}\int_{\Omega} \rho^{\sigma}_{\epsilon} f^{\beta}(\rho^{\sigma}_{\epsilon}) dx dt \right)^{\frac{1}{2}}  \left(\int_{0}^{T}\int_{\Omega}   \rho^{\sigma}_{\epsilon} f^{\beta}(\rho^{\sigma}_{\epsilon}) |\nabla \phi^{\sigma}_{\epsilon} |^{2} \right)^{\frac{1}{2}} check
 \end{aligned}
 \end{equation}
Therefore, from \eqref{EstimateA}, we can conclude that $\rho^{\sigma}_{\epsilon}\mathcal{H}_{p}(x,\mathcal{T}_{\epsilon}\rho^{\sigma}_{\epsilon},\nabla \phi^{\sigma}_{\epsilon}) \to \rho^{\sigma}\mathcal{H}_{p}(x,\rho^{\sigma},\nabla \phi^{\sigma})$ in $L^{1}(\mathrm{Q}_{T})$. As consequence, \eqref{weakrho} holds for any $\varphi \in C^{\infty}_{c}((0,T)\times\Omega)$. Moreover, we can apply \cite[Theorem 6.1]{porretta2015weak}, to deduce the strong convergence of $\rho^{\sigma}_{\epsilon}$ in $C([0,T];L^{1}(\Omega))$.

\end{proof}
\subsection{Existence and uniqueness for the system  of \eqref{EqFluidHd0x}-\eqref{EikonalInx}}
The objective of this section is to establish the existence and uniqueness of a weak solution for the system given in \eqref{EqFluidHd0x}-\eqref{EikonalInx}.

\begin{lemma}\label{lemma4}
Let $(\rho^{\sigma},\phi^{\sigma})$ be given by Theorem \ref{Thmlim}. Then 
 \begin{equation}\label{WeakEq}
\int_{0}^{T}\int_{\Omega}\phi^{\sigma}\partial_{t}\varphi dx dt - \sigma \int_{0}^{T}\int_{\Omega} \phi^{\sigma}\Delta\varphi  dx dt  + \int_{0}^{T}\int_{\Omega} \mathcal{H}(x,\rho^{\sigma},\nabla \phi^{\sigma}) \varphi dx dt  \le  \int_{0}^{T}\int_{\Omega}\phi_{T}\varphi(T) dx
\end{equation}
for every nonnegative function $\varphi \in C_{c}^{\infty}((0,T]\times\Omega)$. Moreover, we have $\phi_{T}\rho^{\sigma}(T) \in L^{1}(\Omega)$ and 
\[
\int_{\Omega}\phi_{T}\rho^{\sigma}(T)  dx \le C
\]
for some constant  $C$ depending on $\phi_{T}$, $\rho_{m}$, $\beta$, and $\rho_{0}$.
\end{lemma}
\begin{proof}
From \eqref{EstimateA001} and the convergence of $\rho^{\sigma}$, we can apply Fatou's lemma and we get for any $\varphi \in C_{c}^{\infty}((0,T]\times\Omega)$, $\varphi\ge 0$,
\[
 \int_{0}^{T}\int_{\Omega} \mathcal{H}(x,\rho^{\sigma},\nabla \phi^{\sigma}) \varphi dx dt  \le  \liminf_{\epsilon \to 0 }\int_{0}^{T}\int_{\Omega} \mathcal{H}(x,\mathcal{T}_{\epsilon}\rho_{\epsilon}^{\sigma},\nabla \phi_{\epsilon}^{\sigma}) \varphi dx dt.
\]
Now, we will deal with the boundary condition at $t=T$, by  \ref{EstimateA}  we obtain
\[
\lim_{\epsilon \to 0 }\int_{\Omega}\phi_{\epsilon}^{\sigma}(T,x)\varphi(T,x)dx =\lim_{\epsilon \to 0 }\int_{\Omega}\phi_{\epsilon}^{\sigma}(T,x)\varphi_{\epsilon}(T,x)dx =\int_{\Omega}\phi_{T}(x)\varphi(T,x)dx
\]
Combining the previous estimates, we deduce \eqref{WeakEq}. 
\end{proof}
\begin{lemma}\label{lemmaA}
Let $(\rho_{\epsilon}^{\sigma},\phi_{\epsilon}^{\sigma})$ be given by  Theorem \ref{Thmlim}. Then, we can extract a subsequence, ${\phi_{\epsilon}^{\sigma}}_{|_{t=0}}$  converges weakly $^{\ast}$ to a bounded measure $\bar{\bar{\phi}}_{0}$  on $\Omega$ and $ \phi^{\sigma}(0) \ge \bar{\bar{\phi}}_{0} $.  

\end{lemma} 

For the proof of Lemma \ref{lemmaA}, we send the reader to \cite[Lemma 4.7]{achdou2018mean}.
Now we give a useful lemma for the uniqueness results. 
\begin{lemma}\label{lemmaIneq}
Under the assumptions \text{{\bf (AS)}}, let $\phi^{\sigma} \in L^{\infty}(0,T;L^{1}(\Omega))$ satisfies \eqref{WeakEq} and let $\rho^{\sigma}$ be a solution of \eqref{EqFluidHd0x} and we assume that $(\rho^{\sigma},\phi^{\sigma})$ satisfy Definition \ref{defweaksig}. Then, we have 
 \begin{equation}\label{EqX0}
 \langle\phi^{\sigma}(0),\rho_{0}\rangle \le \int_{\Omega} \phi_{T} \tilde{\rho}^{\sigma}(T)dx  + \int_{0}^{T}\int_{\Omega} \tilde{\rho}\mathcal{H}_{p}(x,\tilde{\rho}^{\sigma},\nabla \tilde{\phi}^{\sigma})\cdot \nabla \phi^{\sigma} -\tilde{\rho}^{\sigma}\mathcal{H}(x,\rho^{\sigma},\nabla \phi^{\sigma})
\end{equation}
for any $(\tilde{\rho}^{\sigma},\tilde{\phi}^{\sigma})$ verifying the same condition as $(\rho^{\sigma},\phi^{\sigma})$. 
\end{lemma}

\begin{proof}
Let $m_{\delta}(.)$ be s sequence of standard symmetric mollifiers in $\mathbb{R}^{2}$ and set
\[
\tilde{\rho}^{\sigma}_{\delta}(t,x)=\tilde{\rho}^{\sigma}_{\delta}\star m_{\delta} = \int_{\mathbb{R}^{2}} \tilde{\rho}^{\sigma}_{\delta}(t,y) m_{\delta}(x-y) dy.
\]
Since $\tilde{\rho}^{\sigma} \in L^{\infty}(0,T;L^{1}(\Omega))$ ,we deduce that $\tilde{\rho}^{\sigma}_{\delta}$, $\nabla \tilde{\rho}^{\sigma}_{\delta} \in L^{\infty}(\mathrm{Q}_{T})$. We consider a sequence of $1-d$ mollifiers $\xi_{\epsilon}(t)$ such that $\text{supp} (\xi_{\epsilon}) \subset (-\epsilon,0)$, and we set
\[
\tilde{\rho}^{\sigma}_{\delta,\epsilon} := \int_{0}^{T}\xi_{\epsilon}(s-t)\tilde{\rho}^{\sigma}_{\delta}(s) ds
\]
Notice that this function vanishes at $t=0$, so we can consider it as a test function in the inequality satisfied by $\phi^{\sigma}$. We get 
\begin{equation}\label{EsA}
\int_{0}^{T}\int_{\Omega}\phi^{\sigma} \left(\partial_{t}\tilde{\rho}^{\sigma}_{\delta,\epsilon} - \sigma\Delta \tilde{\rho}^{\sigma}_{\delta,\epsilon}  \right) dx dt + \int_{0}^{T}\int_{\Omega}\mathcal{H}(x,\rho^{\sigma},\nabla \phi^{\sigma})\tilde{\rho}^{\sigma}_{\delta,\epsilon} dx dt  \le \int_{\Omega} \phi_{T} \tilde{\rho}^{\sigma}_{\delta,\epsilon} dx 
\end{equation}
Let ${\phi^{\sigma}_{\delta,\epsilon}(s,y)= \int_{0}^{T}\int_{\Omega} \phi^{\sigma}_{\delta,\epsilon}\xi_{\epsilon}(s-t) m_{\delta}(x-y) dt dx }$, then 
\[
\int_{0}^{T}\int_{\Omega}\phi^{\sigma} \left(\partial_{t}\tilde{\rho}^{\sigma}_{\delta,\epsilon} - \sigma\Delta \tilde{\rho}^{\sigma}_{\delta,\epsilon}  \right) dx dt = \int_{0}^{T}\int_{\Omega} \left(-\partial_{s}\phi^{\sigma}_{\delta,\epsilon} - \sigma\Delta \phi^{\sigma}_{\delta,\epsilon} \right) \tilde{\rho}^{\sigma}_{\delta,\epsilon} (s,y)dy ds.
\]
Moreover, from the equation of $\tilde{\rho}^{\sigma}$ we obtain 
\begin{equation}
\begin{aligned}
\int_{0}^{T}\int_{\Omega} &\left(-\partial_{s}\phi^{\sigma}_{\delta,\epsilon} - \sigma\Delta \phi^{\sigma}_{\delta,\epsilon} \right) \tilde{\rho}^{\sigma}_{\delta,\epsilon} (s,y)dy ds \\& =-\int_{0}^{T}\int_{\Omega}  \tilde{\rho}^{\sigma}_{\delta,\epsilon} (s,y)  \tilde{b}_{\delta,\epsilon} (s,y)\cdot \nabla \phi^{\sigma}_{\delta,\epsilon} dy ds + \int_{\Omega}\tilde{\rho}^{\sigma}_{0}(y)\phi^{\sigma}_{\delta,\epsilon}(0) dy
\end{aligned}
\end{equation}
where $\tilde{b}_{\delta,\epsilon} = \mathcal{H}_{p}(x,\tilde{\rho}^{\sigma},\nabla \tilde{\phi}^{\sigma})$. We have 
\begin{equation}
\begin{aligned}
-\int_{0}^{T}\int_{\Omega}\nabla \phi^{\sigma} \cdot \tilde{w}_{\delta,\epsilon} &+  \int_{0}^{T}\int_{\Omega}\mathcal{H}(x,\rho^{\sigma},\nabla \phi^{\sigma})\tilde{\rho}^{\sigma}_{\delta,\epsilon} dx dt  + \int_{\Omega} \tilde{\rho}^{\sigma}_{0}\star m_{\delta}\int_{0}^{T}\phi^{\sigma}(t)\xi_{\epsilon}(-t) dt dx \\& \le  \int_{\Omega}\phi_{T} \rho^{\sigma}_{\delta,\epsilon}  dx
\end{aligned}
\end{equation}
where here we denote $\tilde{w}_{\delta} = \left(\left(\tilde{b}\tilde{\rho}^{\sigma}\right)\star m_{\delta} \right)$ and $\tilde{w}_{\delta,\epsilon}= \int_{0}^{T} \tilde{w}_{\delta}(s)\xi_{\epsilon}(s-t) ds$.
Now from the definition \ref{defweaksig}, we can use Fatou's lemma, and we deduce
\begin{equation}
\begin{aligned}
\liminf_{\epsilon \to 0}& \int_{0}^{T}\int_{\Omega} \left(\mathcal{H}(x,\rho^{\sigma},\nabla \phi^{\sigma})\tilde{\rho}^{\sigma}_{\delta,\epsilon} - \nabla \phi^{\sigma} \cdot \tilde{w}_{\delta,\epsilon} \right)dx dt\\ &
=\liminf_{\epsilon \to 0} \int_{0}^{T}\int_{\Omega}  \int_{0}^{T} \left(\mathcal{H}(x,\rho^{\sigma},\nabla \phi^{\sigma})\tilde{\rho}^{\sigma}_{\delta} - \nabla \phi^{\sigma} \cdot \tilde{w}_{\delta} \right) \xi_{\epsilon}(s-t) ds dx dt\\ \qquad &  \ge  \int_{0}^{T}\int_{\Omega} \left(\mathcal{H}(x,\rho^{\sigma},\nabla \phi^{\sigma})\tilde{\rho}^{\sigma}_{\delta} - \nabla \phi^{\sigma} \cdot \tilde{w}_{\delta} \right) dx dt
\end{aligned}
\end{equation}
and in the same way we get 
\begin{equation}\label{EqRes0}
\begin{aligned}
\liminf_{\delta \to 0}& \int_{0}^{T} \int_{\Omega}  \left(\mathcal{H}(x,\rho^{\sigma},\nabla \phi^{\sigma})\tilde{\rho}^{\sigma}_{\delta} - \nabla \phi^{\sigma} \cdot \tilde{w}_{\delta} \right) dx dt\\& \ge  \int_{0}^{T}\int_{\Omega}  \left(\mathcal{H}(x,\rho^{\sigma},\nabla \phi^{\sigma})\tilde{\rho}^{\sigma} - \nabla \phi^{\sigma} \cdot\mathcal{H}(x,\rho^{\sigma},\nabla \phi^{\sigma})\tilde{\rho}^{\sigma}  \right)dx dt
\end{aligned}
\end{equation}
For $t=0$, we have $\int_{\Omega} \phi^{\sigma}(t)(\tilde{\rho_{0}}\star m_{\delta}) dx $ has a trace at $t=0$ from Lemma \ref{lemma33} and this trace is continuous as $\delta \to 0$ since $\tilde{\rho}_{0}$ is continuous. Then as $\epsilon \to 0 $, by \eqref{EsA} we get 
\begin{equation}\label{EqRes}
\begin{aligned}
 \int_{0}^{T} \int_{\Omega}&\left(\mathcal{H}(x,\rho^{\sigma},\nabla \phi^{\sigma})\tilde{\rho}^{\sigma}_{\delta} - \nabla \phi^{\sigma} \cdot \tilde{w}_{\delta} \right) dx dt + \int_{\Omega} \phi^{\sigma}(t)(\tilde{\rho_{0}}\star m_{\delta}) dx \le & \int_{\Omega}\phi_{T}\star m_{\delta}\tilde{\rho}_{\delta}(T)dx 
 \end{aligned}
\end{equation}
Passing now to the limit in \eqref{EqRes} and using \eqref{EqRes0} we deduce \eqref{EqX0}.
\end{proof}

\begin{lemma}
Consider a subsequence $(\rho^{\sigma}_{\epsilon},\phi^{\sigma}_{\epsilon})$  converging to $(\rho^{\sigma},\phi^{\sigma})$ as in Theorem \ref{Thmlim}. Then, $(\rho^{\sigma},\phi^{\sigma})$ satisfies the following energy identity
 \begin{equation}\label{Eqenerg}
\langle\rho^{\sigma}(T),\phi_{T}\rangle + \int_{0}^{T}\int_{\Omega} \rho\mathcal{H}_{p}(x,\rho^{\sigma},\nabla \phi^{\sigma})\cdot \nabla \phi^{\sigma} -\mathcal{H}(x,\rho^{\sigma},\nabla \phi^{\sigma})= \langle\phi^{\sigma}(0),\rho_{0}\rangle
\end{equation}
\end{lemma}
\begin{proof}
We star by the energy identity for the system  \eqref{EqFluidHd0xa}-\eqref{EikonalInxa}
 \begin{equation}\label{EqX}
\int_{\Omega}\phi^{\sigma}_{\epsilon}(0) \rho_{\epsilon 0} dx = \int_{\Omega}\phi_{\epsilon T}\rho^{\sigma}_{\epsilon}(T)dx + \int_{0}^{T}\int_{\Omega} \rho^{\sigma}_{\epsilon}\left(\mathcal{H}_{p}(x,\mathcal{T}_{\epsilon}\rho^{\sigma}_{\epsilon},\nabla \phi^{\sigma}_{\epsilon})\cdot \nabla \phi^{\sigma}_{\epsilon} -\mathcal{H}(x,\mathcal{T}_{\epsilon}\rho^{\sigma}_{\epsilon},\nabla \phi^{\sigma}_{\epsilon}) \right) dx dt.
\end{equation}
 Using ${\left(\mathcal{H}_{p}(x,\rho, p)\cdot p -\mathcal{H}(x,\rho, p) \right) \ge 0}$, estimate \eqref{EstimateA} and Fatou's lemma ,we get 
 \begin{equation}\label{EqX1}
 \begin{aligned}
\lim_{\epsilon \to 0}\inf \int_{0}^{T}\int_{\Omega} \rho^{\sigma}_{\epsilon} &\left(\mathcal{H}_{p}(x,\mathcal{T}_{\epsilon}\rho^{\sigma}_{\epsilon},\nabla \phi^{\sigma}_{\epsilon})\cdot \nabla \phi^{\sigma}_{\epsilon} -\mathcal{H}(x,\mathcal{T}_{\epsilon}\rho^{\sigma}_{\epsilon},\nabla \phi^{\sigma}_{\epsilon}) \right) dx dt \\& \ge  \int_{0}^{T}\int_{\Omega} \rho^{\sigma}\left(\mathcal{H}_{p}(x,\rho^{\sigma},\nabla \phi^{\sigma})\cdot \nabla \phi^{\sigma} -\mathcal{H}(x,\rho^{\sigma},\nabla \phi^{\sigma}) \right) dx dt.
\end{aligned}
\end{equation}
From Lemma \ref{lemmaA}, we can assume that ${\phi^{\sigma}_{\epsilon}}_{t=0}$ converges weakly $\star$ to a bounded measure $\bar{\bar{\phi}}_{0}$ on $\Omega$, and $\bar{\bar{\phi}}_{0} \le \phi^{\sigma}(0)$. Since $\rho_{0}\in C(\Omega)$, we deduce 
\[
\lim _{\epsilon \to 0}\int_{\Omega}\phi^{\sigma}_{\epsilon}(0)\rho_{\epsilon 0} dx \le \langle \phi^{\sigma}_{\epsilon}(0),\rho_{0}\rangle.
\] 
Combining  the previous results we get from \eqref{EqX}
 \begin{equation}\label{Eqss}
\begin{aligned}
\limsup_{\epsilon \to 0}&\int_{\Omega}\phi_{T}\rho_{\epsilon}^{\sigma}(T) dx  \le   \langle\phi^{\sigma}(0),\rho_{0}\rangle - \int_{0}^{T}\int_{\Omega} \rho^{\sigma}\left(\mathcal{H}_{p}(x,\rho^{\sigma},\nabla \phi^{\sigma})\cdot \nabla \phi^{\sigma} -\mathcal{H}(x,\rho^{\sigma},\nabla \phi^{\sigma}) \right) dx dt
\end{aligned}
\end{equation}
We use \eqref{EqX0} in Lemma \ref{lemmaIneq} for $\rho^{\sigma}=\tilde{\rho}^{\sigma}$ and $\phi^{\sigma}=\tilde{\phi}^{\sigma}$
 \begin{equation}\label{Eqss}
\begin{aligned}
\limsup_{\epsilon \to 0}\int_{\Omega}\phi_{\epsilon T}\rho_{\epsilon}^{\sigma}(T) dx  \le  & \langle\phi^{\sigma}(0),\rho_{0}\rangle - \int_{0}^{T}\int_{\Omega} \rho^{\sigma}\left(\mathcal{H}_{p}(x,\rho^{\sigma},\nabla \phi^{\sigma})\cdot \nabla \phi^{\sigma} -\mathcal{H}(x,\rho^{\sigma},\nabla \phi^{\sigma}) \right) dx dt\\ & \le \int_{\Omega}\phi_{T}\rho^{\sigma}(T) dx 
\end{aligned}
\end{equation}
Then, we can conclude \eqref{Eqenerg}.
\end{proof}
Now, we will give the existence and uniqueness results 
\begin{proposition}\label{propoexun}
Consider a subsequence $(\rho^{\sigma}_{\epsilon},\phi^{\sigma}_{\epsilon})$ converging to the $(\rho^{\sigma}, \phi^{\sigma} )$ as in Theorem \ref{Thmlim}. Then, $(\rho^{\sigma}, \phi^{\sigma} )$ is a weak solution of  \eqref{EqFluidHd0x}-\eqref{EikonalInx}. Moreover, for any weak solution $(\rho^{\sigma}, \phi^{\sigma} )$ to \eqref{EqFluidHd0x}-\eqref{EikonalInx}, the energy identity \eqref{Eqenerg} remains true. Furthermore, under the following monotonicity condition  
 \begin{equation}\label{Ourassumption}
\mathbf{H}_{c}=\begin{pmatrix}
-\frac{2}{\rho}\mathcal{H}_{\rho}(x,\rho^{\sigma}, p)& \mathcal{H}_{\rho,p}^{T}( x,\rho^{\sigma}, p )\\
\mathcal{H}_{\rho,p}( x,\rho^{\sigma}, p )& 2\mathcal{H}_{p,p}( x,\rho^{\sigma}, p )\\
\end{pmatrix} 
 \ge 0, \quad \rho>0,
\end{equation}
there exists a unique weak solution of the system \eqref{EqFluidHd0x}-\eqref{EikonalInx} in sense of Definition \ref{defweaksig}. In particular,  
 \begin{enumerate}
\item  If $\beta=2$ and ${\rho^{\sigma} \le \frac{\rho_{m}}{2}}$, there exists a unique weak solution of the system \eqref{EqFluidHd0x}-\eqref{EikonalInx}.
\item If $\beta=0$ the monotonicity condition \eqref{Ourassumption} is not fulfilled.
\end{enumerate}
\end{proposition}
\begin{proof}
By Theorem \ref{Thmlim}, we deduce that $(\rho^{\sigma}, \phi^{\sigma} )$ is a weak solution of  \eqref{EqFluidHd0x}-\eqref{EikonalInx}. Moreover, from Lemma \ref{lemmaIneq}
 \begin{equation}\label{EqX0a}
 \langle\phi^{\sigma}(0),\rho_{0}\rangle \le \int_{\Omega} \phi_{T} \rho^{\sigma}(T)dx  + \int_{0}^{T}\int_{\Omega} \rho\mathcal{H}_{p}(x,\rho^{\sigma},\nabla \phi^{\sigma})\cdot \nabla \phi^{\sigma} -\rho^{\sigma}\mathcal{H}(x,\rho^{\sigma},\nabla \phi^{\sigma})
\end{equation}
Now , let us consider $\phi_{k}:= \min(\phi^{\sigma},k)$. We have $-\partial_{t} \phi^{\sigma} - \sigma \Delta \phi^{\sigma} \in  L^{1}(\mathrm{Q}_{T})$ and by Kato’s inequality we obtain 
\begin{equation}
-\partial_{t} \phi_{k} - \sigma \Delta \phi_{k} + \mathcal{H}(x,\rho^{\sigma}_{\epsilon},\nabla \phi^{\sigma}_{\epsilon}) 1_{\{\phi^{\sigma}<k \}} \ge 0.
\end{equation}
Since, $\partial_{t}\rho^{\sigma} - \sigma \Delta \rho^{\sigma} - \div(\rho^{\sigma}b ) =0$ for some $b$ such that $\rho^{\sigma}|b|^2 \in  L^{1}(\mathrm{Q}_{T})$, by \cite[Theorem 3.6]{porretta2015weak}, we have $\rho^{\sigma} \in C^{0}(0,T;L^{1}(\Omega))$ is also a normalized solution and we have 
\[
\partial_{t}S_{n}(\rho^{\sigma}) - \sigma \Delta S_{n}(\rho^{\sigma}) - div(S^{'}_{n}(\rho^{\sigma})\rho^{\sigma}\mathcal{H}_{p}(x,\rho^{\sigma},\nabla \phi^{\sigma} )= R_{n}
\]
where $S_{n}(\rho^{\sigma})$ is a suitable $C^{1}$ truncation and $R_{n} \to 0$ in $L^{1}(\mathrm{Q}_{T})$, for more details see \cite{porretta2015weak}. Now, we multiply the equation of $S_{n}(\rho^{\sigma}) $ by $\phi_{k}$ we get 
 \begin{equation}
\begin{aligned}
\int_{\Omega} S_{n}(\rho_{0})\phi_{k}(0) dx - & \int_{\Omega} S_{n}(\rho^{\sigma}(T))\phi_{k}(T) dx \ge - \int_{0}^{T}\int_{\Omega} R_{n} \phi_{k} dx dt \\ & \int_{0}^{T}\int_{\Omega} \left(S^{'}_{n}(\rho^{\sigma})\rho^{\sigma}\mathcal{H}_{p}(x,\rho^{\sigma},\nabla \phi^{\sigma})\cdot \nabla \phi^{\sigma} -S_{n}(\rho^{\sigma})\mathcal{H}(x,\rho^{\sigma},\nabla \phi^{\sigma})\right)\mathrm{1}_{\{\phi^{\sigma}<k \}} dx dt 
\end{aligned}
\end{equation}
We have $\phi_{k}$ is bounded then the first term in the right-hand side vanishes as $n \to +\infty$. Then we pass to the limit as  $n \to +\infty$, we obtain 
 \begin{equation}
\begin{aligned}
\int_{\Omega} \rho_{0}\phi_{k}(0) dx - & \int_{\Omega} \rho^{\sigma}(T)\phi_{k}(T) dx \\ & \ge \int_{0}^{T}\int_{\Omega} \left(\rho^{\sigma}\mathcal{H}_{p}(x,\rho^{\sigma},\nabla \phi^{\sigma})\cdot \nabla \phi^{\sigma} -\rho^{\sigma}\mathcal{H}(x,\rho^{\sigma},\nabla \phi^{\sigma})\right) \mathrm{1}_{\{\phi^{\sigma}<k \}}  dx dt 
\end{aligned}
\end{equation}
Finally, by letting $k \to +\infty$  and with \eqref{EqX0a} we deduce that the energy  identity \eqref{Eqenerg} remains true.

We now process for proving the uniqueness results. Let $(\rho^{}, \phi)$ and  $(\tilde{\rho}, \tilde{\phi})$ be two solution of \eqref{EqFluidHd0x}-\eqref{EikonalInx} both solutions satisfy the energy identity \eqref{Eqenerg}. 

First, using the inequality \eqref{EqX0}  from Lemma \ref{lemmaIneq} we get 
 \begin{equation*}
\langle\phi(0),\rho_{0}\rangle \le \langle\tilde{\rho}(T),\phi_{T}\rangle + \int_{0}^{T}\int_{\Omega} \tilde{\rho}\mathcal{H}_{p}(x,\tilde{\rho},\nabla \tilde{\phi})\cdot \nabla \phi -\tilde{\rho}\mathcal{H}(x,\rho,\nabla \phi)
\end{equation*}
 \begin{equation*}
 \langle\tilde{\phi}(0),\rho_{0}\rangle \le \langle\rho(T),\phi_{T}\rangle + \int_{0}^{T}\int_{\Omega} \rho\mathcal{H}_{p}(x,\rho,\nabla \phi)\cdot \nabla \tilde{\phi} -\rho\mathcal{H}(x,\tilde{\rho},\nabla \tilde{\phi})
\end{equation*}
Moreover, by the energy identity \eqref{Eqenerg} we obtain 
 \begin{equation*}
 \langle\phi(0),\rho_{0}\rangle =\langle\rho(T),\phi_{T}\rangle + \int_{0}^{T}\int_{\Omega} \rho\mathcal{H}_{p}(x,\rho,\nabla \phi)\cdot \nabla \phi -\rho\mathcal{H}(x,\rho,\nabla \phi)
\end{equation*}
 \begin{equation*}
\langle\tilde{\phi}(0),\rho_{0}\rangle =\langle\tilde{\rho}(T),\phi_{T}\rangle + \int_{0}^{T}\int_{\Omega} \tilde{\rho}\mathcal{H}_{p}(x,\tilde{\rho},\nabla \tilde{\phi})\cdot \nabla \tilde{\phi} -\rho\mathcal{H}(x,\tilde{\rho},\nabla \tilde{\phi}),
\end{equation*}
and therefore we get 
 \begin{equation}
 \begin{aligned}
\int_{0}^{T}\int_{\Omega} &\tilde{\rho}\left(\mathcal{H}(x,\rho,\nabla \phi)-\mathcal{H}(x,\tilde{\rho},\nabla \tilde{\phi})-\mathcal{H}_{p}(x,\tilde{\rho},\nabla \tilde{\phi})\cdot \left(\nabla \phi-\nabla \tilde{\phi}\right) \right) dx dt\\
+&\int_{0}^{T}\int_{\Omega} \rho\left(\mathcal{H}(x,\tilde{\rho},\nabla \tilde{\phi})-\mathcal{H}(x,\rho,\nabla \phi )-\mathcal{H}_{p}(x,\rho,\nabla \phi)\cdot \left(\nabla \tilde{\phi}-\nabla \phi\right) \right) dx dt =0.
\end{aligned}
\end{equation}
Let us define 
 \begin{equation}
 \begin{aligned}
E(x,\rho_{1},p_{1},\rho_{2},p_{2}) =&-\left( \mathcal{H}(x,\rho_{1}, p_{1}) - \mathcal{H}(x,\rho_{2}, p_{2})\right)\left( \rho_{1} - \rho_{2}\right) \\
& +\left( \rho_{1} \mathcal{H}_{p}(x,\rho_{1}, p_{1}) -\rho_{2} \mathcal{H}_{p}(x,\rho_{2}, p_{2}) \right)\left(p_{1} - p_{2}\right)\\
&=\begin{pmatrix} \rho_{2}- \rho_{1}\\ p_{2}- p_{1}
 \end{pmatrix}\cdot \begin{pmatrix} \mathcal{H}(x,\rho_{1}, p_{1}) - \mathcal{H}(x,\rho_{1}, p_{1})\\  \rho_{2} \mathcal{H}_{p}(x,\rho_{2}, p_{2}) -\rho_{1} \mathcal{H}_{p}(x,\rho_{1}, p_{1})
 \end{pmatrix}\\&=\begin{pmatrix} \rho_{2}- \rho_{1}\\ p_{2}- p_{1}
 \end{pmatrix}\cdot \left( \Theta\begin{pmatrix} \rho_{2}\\ p_{2} \end{pmatrix} -\Theta \begin{pmatrix} \rho_{1}\\ p_{1} \end{pmatrix}\right)
\end{aligned}
\end{equation}
with $\Theta\begin{pmatrix} \rho\\ p \end{pmatrix} =\begin{pmatrix} - \mathcal{H}(x,\rho, p)\\ \rho \mathcal{H}_{p}(x,\rho, p) \end{pmatrix}$. Therefore, the positivity  of $E$ is guaranteed if and only if $\Theta$  monotone which is equivalent to \eqref{Ourassumption}. The preceding monotonicity condition aligns with the condition \eqref{uniq}, which is also the  Lasry–Lions monotonicity condition used by P.L. Lions in his  work on solution uniqueness, as demonstrated in \cite{Lions2023}.

 In addition, assuming that all the following differentiations are allowed, we see that the Hessian 
  \begin{equation*}
 \mathbf{H}_{c}=\begin{pmatrix}
-\frac{2}{\rho}\mathcal{H}_{\rho}(x,\rho^{\sigma}, p)& \mathcal{H}_{\rho,p}^{T}( x,\rho^{\sigma}, p )\\
\mathcal{H}_{\rho,p}( x,\rho^{\sigma}, p )& 2\mathcal{H}_{p,p}( x,\rho^{\sigma}, p )\\
\end{pmatrix} 
\end{equation*}
and 
 $\mathcal{H}_{\rho}(x,\rho,  p ) =- \frac{\beta}{2} \left(\rho_{m}-\rho\right)^{\beta-1}p^2 + \frac{2-\beta}{2}\left(\rho_{m}-\rho \right)^{1-\beta}$, $
\mathcal{H}_{p}(x,\rho, p ) = \left(\rho_{m}-\rho\right)^{\beta}p$, $\mathcal{H}_{\rho,p}(x,\rho, p ) =-\beta \left(\rho_{m}-\rho\right)^{\beta-1}p$ and $\mathcal{H}_{p,p}(x,\rho,p ) =  \left(\rho_{m}-\rho\right)^{\beta}$. 

For $\rho>0$, the matrix $\mathbf{H}_{c}$ is given by 
\[
\mathbf{H}_{c}=\begin{pmatrix}
-\frac{2}{\rho}\left( - \frac{\beta}{2} \left(\rho_{m}-\rho\right)^{\beta-1}p^2 + \frac{2-\beta}{2}\left(\rho_{m}-\rho \right)^{1-\beta} \right)& -\beta \left(\rho_{m}-\rho\right)^{\beta-1}p\\
-\beta \left(\rho_{m}-\rho\right)^{\beta-1}p)&  2 \left(\rho_{m}-\rho\right)^{\beta}\
\end{pmatrix}.
\]
Taking $z=(z_{1},z_{2}) \in \mathbb{R}^{2}$. Therefore, the matrix $\mathbf{H}_{c}$ is positive semi-definite  if and only if 
\[
z^{T}\mathbf{H}_{c}z = -\frac{2}{\rho}\left( - \frac{\beta}{2} \left(\rho_{m}-\rho\right)^{\beta-1}p^2 + \frac{2-\beta}{2}\left(\rho_{m}-\rho \right)^{1-\beta} \right) z_{1}^{2} -2\beta \left(\rho_{m}-\rho\right)^{\beta-1}p z_{1}z_{2} +2 \left(\rho_{m}-\rho\right)^{\beta}z_{2}^{2} \ge 0.
\]

In particular, when we consider the case of $\beta =2$, we have 
\begin{equation}
\begin{aligned}
z^{T}\mathbf{H}_{c} z = &\frac{2}{\rho}\left( \left(\rho_{m}-\rho\right)p^2  \right) z_{1}^{2} -4 \left(\rho_{m}-\rho\right)p z_{1}z_{2} +2 \left(\rho_{m}-\rho\right)^{2}z_{2}^{2}  \\
 = &\frac{2}{\rho} \left(\rho_{m}-\rho\right)p^2  z_{1}^{2} -2p^{2} z_{1}^{2} +2\left(p z_{1} -\left(\rho_{m}-\rho\right)z_{2}\right)^{2}\\
 = &2p^{2} z_{1}^{2}\left(\frac{\left(\rho_{m}-\rho\right)}{\rho} -1\right) +2\left(p z_{1} -\left(\rho_{m}-\rho\right)z_{2}\right)^{2}
\end{aligned}
\end{equation}
Consequently, the assumption $0 <\rho \le \frac{\rho_{m}}{2}$ is a sufficient condition to  guarantee that  the matrix $\mathbf{H}_{c}$ is positive semi-definite. Now, when we consider the case of $\beta =0$, we have 
\[
z^{T}\mathbf{H}_{c} z = -\frac{2}{\rho}\left(\rho_{m}-\rho  \right) z_{1}^{2} +2 z_{2}^{2},
\]
which is not positive for all $z \in \mathbb{R}^{2}$. Finally, we conclude with the desired uniqueness results.
\end{proof}
\begin{remark}
The Larsy-Lions monotonicity condition suggests that pedestrians tend to avoid densely populated areas and prefer to be more spread out. This is evident in our model when $\beta = 2$, as pedestrians actively avoid high-density regions. The condition $0 < \rho \le \frac{\rho_{m}}{2}$ is essential for the uniqueness of the solution, which is logical as people tend to distribute themselves more evenly when an area is close to half its capacity. This condition not only ensures that the solution is unique, but also that it is stable against perturbations, as discussed in \cite{GMP2023preprint}.
\end{remark}

\subsection{Vanishing viscosity limit }
In this section, we will analyze the vanishing viscosity limit of weak solutions. Here, we will utilize the connection between the weak sub-solution for the HJB equation and the weak solution for the Kolmogorov-type equation. The analysis in this section was inspired by \cite[Section 1.3.7, p.83]{achdou2020introduction}
\begin{e-definition}
Under the assumptions \text{{\bf (AS)}}, Given $h\in L^{1}(\mathrm{Q}_{T})$, and $\rho\in L^{1}(\mathrm{Q}_{T})$. A function $\phi \in L^{\infty}(0,T;L^{1}(\Omega)) \cap L^{q}(0,T;W^{1,q}(\Omega))$ for every $q<\frac{4}{3}$, is a weak sub-solution  of 
\begin{align}
-\partial_t\phi  -   \mathcal{H}(x,\rho,\nabla \phi) =h, & \label{HJBEq0}\\
\phi(t=T,x)= \phi_{T}(x),& \label{HJBEqb}
\end{align}
if it satisfies 
\begin{equation}
\begin{aligned}
\int_{0}^{T}\int_{\Omega} \phi \partial_{t}\varphi dx dt + \int_{0}^{T}\int_{\Omega}  \mathcal{H}(x,\rho,\nabla \phi) \varphi dx dt  \le & \int_{0}^{T}\int_{\Omega}  h \varphi dx dt +\int_{\Omega} \phi_{T}(x)\varphi(T) dx \\& \forall  \varphi \in C^{1}_{c}((0,T]\times \Omega),\,\,\, \varphi \ge 0
\end{aligned}
\end{equation}
We will shortly denotes   $-\partial_t\phi  + \mathcal{H}(x,\rho,\nabla \phi) \le h $ and $\phi(T) \le  \phi_{T}$, the definition of a weak sub-solution.
\end{e-definition}
Now we state the existence of a weak solution for the Kolmogorov-type equation. We denote by $\mathcal{P}(\Omega)$ the set of Borel probability measures on $\Omega$.
\begin{e-definition}
Let $\rho_{0} \in \mathcal{P}(\Omega)$, Given a measurable vector field $b : \mathrm{Q}_{T} \longrightarrow \mathbb{R}$ , a function $\rho \in  L^{1}(\mathrm{Q}_{T})$, is a weak solution of 
\begin{equation}\label{contEq0}
\begin{cases}
\partial_t \rho  -   div(\rho b) =0, & \\\
\rho(t=0,x)= \rho_{0}(x),&
\end{cases}
\end{equation}
if  $\rho \in C^{0}([0,T]; \mathcal{P}(\Omega))$, $\int_{0}^{T}\int_{\Omega}\rho|b|^{2} dx dt <\infty$ and we have 
\begin{equation}\label{WeakfomulationA}
-\int_{0}^{T}\int_{\Omega}\rho \partial_{t}\varphi dx dt + \int_{0}^{T}\int_{\Omega} \rho b\cdot \nabla\varphi dx dt = \int_{\Omega}\rho_{0}\varphi(0) \quad \forall  \varphi \in C^{1}_{c}([0,T)\times \Omega).
\end{equation}
\end{e-definition}
Now, we give a definition of the weak solution for the first-order MFGs system. 

\begin{e-definition} \label{WeakSubSolSystem}
Under the assumptions \text{{\bf (AS)}}, a  pair $(\rho, \phi) \in L^{1}(\mathrm{Q}_{T})\times L^{\infty}(0,T;L^{1}(\Omega))$ is a weak solution to the following  first-order MFG system 
\begin{align}
\partial_t\rho  - \nabla \cdot \left( \rho \mathcal{H}_{p}(x, \rho,\nabla \phi)\right)  = 0, &\qquad (t,x)\in\mathrm{Q}_{T}  \label{EqFluidHd0sb}\\ 
\rho(t=0,x)= \rho_{0}(x),&\quad x\in\Omega  \label{FluidInsb}
\end{align}
\begin{align}
-\partial_t\phi  + \mathcal{H}(x, \rho,\nabla \phi ) =0, &\qquad (t,x)\in\mathrm{Q}_{T}  \label{EqEikonalHd0sb}\\
\phi(t=T,x)= \phi_{T}(x),&\quad x\in\Omega \label{EikonalInsb}
\end{align}
if 
\begin{description}
\item[ (i)]  $\phi_{T}\rho(T) \in  L^{1}(\mathrm{Q}_{T})$, $\rho f^{\beta}(\rho) |\nabla \phi|^{2} \in L^{1}(\mathrm{Q}_{T})$ and  $f^{\beta}(\rho) |\nabla \phi|^{2} \in L^{1}(\mathrm{Q}_{T})$
\item[ (ii)]  $\phi$ is a weak sub-solution of \eqref{EqEikonalHd0sb}-\eqref{EikonalInsb}, $\rho \in C^{0}(o,T;\mathcal{P}(\Omega))$ is a weak solution  of \eqref{EqFluidHd0sb}-\eqref{FluidInsb}.
\item[ (iii)] $\phi$ and $\rho$ satisfy the following identity
\begin{equation}\label{WeakfomulationF}
\int_{\Omega} \rho_{0}\phi(0) dx =  \int_{\Omega} \rho(T)\phi_{T} dx + \int_{0}^{T}\int_{\Omega} \rho \left(  \mathcal{H}_{p}(x, \rho ,\nabla \phi)\cdot \nabla \phi -\mathcal{H}(x, \rho,\nabla \phi )\right) dx dt ,
\end{equation}
\end{description}
\end{e-definition}

\begin{lemma}
Under assumptions \text{{\bf (AS)}}, let $\phi$ be a weak subsolution of \eqref{HJBEq0}-\eqref{HJBEqb}and $\rho$ be a weak solution of \eqref{contEq0}. Then, we have $\rho |\nabla \phi |^{2} \in L^{1}(\mathrm{Q}_{T})$ and $\rho(0)\phi(0) \in L^{1}(\Omega)$ and 
\begin{equation}\label{EsR}
\int_{\Omega} \rho_{0}\phi(0) dx \le \int_{\Omega} \phi_{T}\rho(T) dx  + \int_{\Omega}\rho \left(-\nabla \phi \cdot b  + \mathcal{H}(x, \rho ,\nabla \phi)\right) dx dt.
\end{equation}
\end{lemma}
\begin{proof}
Let $m_{\delta}(.)$ be s sequence of standard symmetric mollifiers in $\mathbb{R}^{2}$ and we set $\rho_{\delta}(t,x)= \rho(t,.)\star m_{\delta}$. We consider a sequence of $1-d$ mollifiers $\xi_{\epsilon}(t)$ such that $\text{supp} (\xi_{\epsilon}) \subset (-\epsilon,0)$, and we set
\[
\rho_{\delta,\epsilon} := \int_{0}^{T}\xi_{\epsilon}(s-t)\rho_{\delta}(s) ds.
\]

We have
\begin{equation}\label{EQS}
\int_{0}^{T}\int_{\Omega}\phi \partial_{t}\rho_{\delta,\epsilon} dx dt + \int_{0}^{T}\int_{\Omega}\mathcal{H}(x, \rho,\nabla \phi ) \rho_{\delta,\epsilon} dx dt \le \int_{\Omega}\phi_{T}\rho_{\delta,\epsilon} dx.
\end{equation}

Let ${\phi_{\delta,\epsilon}(s,y)= \int_{0}^{T}\int_{\Omega} \phi_{\delta,\epsilon}\xi_{\epsilon}(s-t) m_{\delta}(x-y) dx dt  }$, then 
\[
\int_{0}^{T}\int_{\Omega}\phi^{\sigma}\partial_{t}\tilde{\rho}_{\delta,\epsilon} dx dt = \int_{0}^{T}\int_{\Omega} \left(-\partial_{s}\phi_{\delta,\epsilon} \right) \tilde{\rho}^{\sigma}_{\delta,\epsilon} (s,y)dy ds.
\]
We observe that $\phi_{\delta,\epsilon}$ vanishes near $s=T$, so it can be used as a test function in \eqref{WeakfomulationA}. It follows that 
\begin{equation}
\begin{aligned}
\int_{0}^{T}\int_{\Omega}&\phi^{\sigma}\partial_{t}\tilde{\rho}_{\delta,\epsilon} dx dt \\ &= -\int_{0}^{T}\int_{\Omega} \partial_{s}\phi_{\delta,\epsilon} \tilde{\rho}^{\sigma}_{\delta,\epsilon} (s,y)dy ds\\&
=- \int_{0}^{T}\int_{\Omega} \rho(s,y) b(s,y)\cdot \nabla\phi_{\delta,\epsilon} dx dt + \int_{\Omega}\rho_{0}(y)\phi_{\delta,\epsilon}(0)
\end{aligned}
\end{equation}
Now we shift the convolution kernels from $\phi$ to $\rho$ and we use it in \eqref{EQS}, we get 
\begin{equation}\label{QA}
\begin{aligned}
-\int_{0}^{T}\int_{\Omega}\nabla \phi \cdot w_{\delta,\epsilon} dx dt &+ \int_{0}^{T}\int_{\Omega}  \mathcal{H}(x, \rho ,\nabla \phi)\rho_{\delta,\epsilon} dx dt + \int_{\Omega} \rho_{0}\star m_{\delta}\left(\int_{0}^{T}\phi(t)\xi_{\epsilon}(-t) dt\right) dx \\& \le \int_{\Omega}\phi_{T}\rho_{\delta,\epsilon}(T) dx
\end{aligned}
\end{equation}
where we denote $w_{\delta}= [(b\rho)\star m_{\delta}]$ and  ${w_{\delta,\epsilon}(s,y)= \int_{0}^{T}w_{\delta}(x)\xi_{\epsilon}(s-t) dt }$.

Since $\phi$ is bounded below, we have 
\[
\liminf_{\delta \to 0}\liminf_{\epsilon \to 0} \int_{\Omega}\rho_{0} \star m_{\delta} \left(\phi(t) \xi_{\epsilon}(-t) dx \right) dx \ge \int_{\Omega} \rho_{0}\phi(0) dx
\]
Moreover we have at $t=T$
\begin{equation}\label{EQestimate0}
\lim_{\epsilon \to 0}\int_{\Omega} \phi_{T}\rho_{\delta,\epsilon}(T) dx =\int_{\Omega} \phi_{T}\rho_{\delta}(T)\star m_{\delta} dx 
\end{equation}

Then, we see that $\phi_{T}$ is bounded below, and then we deduce $|\phi_{T}\rho_{\delta}(T)\star m_{\delta}|$ is dominated in $L^{1}(\Omega)$. Therefore, we have 
\begin{equation}\label{EQestimate}
\lim_{\delta \to 0}\lim_{\epsilon \to 0} \int_{\Omega}\phi_{T}\rho_{\delta,\epsilon}(T) dx =\int_{\Omega} \phi_{T}\rho(T) dx 
\end{equation}

We know that  
\[
\rho_{\delta,\epsilon}\mathcal{H}(x, \rho ,\nabla \phi) = \rho_{\delta,\epsilon}\left(\frac{1}{2}f^{\beta}(\rho)|\nabla \phi|^{2} -\frac{1}{2}f^{2-\beta}(\rho) \right)
\]
and by Young inequality we get
\[
\nabla \phi \cdot w_{\delta,\epsilon} \le \frac{1}{2}\rho_{\delta,\epsilon} |\nabla \phi |^{2} + \frac{1}{2}\frac{|w_{\delta,\epsilon}|^{2}}{\rho_{\delta,\epsilon}},
\]
then 
\[
-\nabla \phi \cdot w_{\delta,\epsilon} + \mathcal{H}(x, \rho ,\nabla \phi) \rho_{\delta,\epsilon} \ge \frac{1}{2}\rho_{\delta,\epsilon}f^{\beta}(\rho) |\nabla \phi |^{2} - \frac{1}{2}\rho_{\delta,\epsilon}f^{2-\beta}(\rho) - \frac{1}{2}\rho_{\delta,\epsilon} |\nabla \phi |^{2} - \frac{1}{2}\frac{|w_{\delta,\epsilon}|^{2}}{\rho_{\delta,\epsilon}}
\]
 Now we introduce the lower semi-continuous function $\Phi$ on $\mathbb{R}^{2}\times\mathbb{R}$ by 
\begin{equation}\label{functionPsi}
\Phi(w,\rho) = \begin{cases}
\frac{|w|^{2}}{\rho} & \text{if}\,\, \rho>0 \\
0  & \text{if}\,\, \rho =0 \,\, \text{and }\, w=0 \\
+\infty \,\,\,\,\,\,\quad \text{otherwise.}
\end{cases}
\end{equation}
We observe that $\Phi $ is convex in the couple $(w,\rho)$. Then ,by Jensen inequality we have ${\frac{|w_{\delta,\epsilon}|^{2}}{\rho_{\delta,\epsilon}} \le  \frac{|w|^{2}}{\rho}} \star m_{\delta,\epsilon}$ ( recalling that $w=b\rho$ hence  $\frac{|w|^{2}}{\rho}= \rho|b|^{2}$), the we deduce that 

\begin{equation}\label{QAs}
-\nabla \phi \cdot w_{\delta,\epsilon} + \mathcal{H}(x, \rho ,\nabla \phi) \rho_{\delta,\epsilon} \ge \frac{1}{2}\rho_{\delta,\epsilon}f^{\beta}(\rho) |\nabla \phi |^{2} - \frac{1}{2}\rho_{\delta,\epsilon}f^{2-\beta}(\rho) - \frac{1}{2}\rho_{\delta,\epsilon} |\nabla \phi |^{2} - \frac{1}{2}\left(\rho|b|^{2}\star m_{\delta,\epsilon}\right)
\end{equation}

From \eqref{QA}-\eqref{EQestimate} we will use Fatou's Lemma we get 
\begin{equation}\label{QA}
\begin{aligned}
\liminf_{\delta \to 0}\liminf_{\epsilon \to 0} & \int_{\Omega}\left(-\nabla \phi \cdot w_{\delta,\epsilon} + \mathcal{H}(x, \rho ,\nabla \phi)\rho_{\delta,\epsilon}\right) dx dt \\& \ge \int_{\Omega}\rho \left(-\nabla \phi \cdot b  + \mathcal{H}(x, \rho ,\nabla \phi)\right) dx dt  
\end{aligned}
\end{equation}
We deduce from \eqref{QAs} that $\rho |\nabla \phi |^{2} \in L^{1}(\mathrm{Q}_{T})$. Then, we obtain \eqref{EsR} by combining  \eqref{EQestimate0} , \eqref{EQestimate} and \eqref{QA}.

\end{proof}

\begin{thm}
Under the assumptions \text{{\bf (AS)}}, let $\left(\rho^{\sigma}, \phi^{\sigma}\right)$ be a weak solution of \eqref{EqFluidHd0x}-\eqref{EikonalInx}. Then, there exists a subsequence and a couple $\left(\rho, \phi\right) \in L^{1}(\mathrm{Q}_{T})\times L^{\infty}(0,T;L^{1}(\Omega))$ such that  $\left(\rho^{\sigma}, \phi^{\sigma}\right) \to \left(\rho, \phi\right)$ in $L^{1}(\mathrm{Q}_{T})$ and $\left(\rho, \phi\right)$ is a weak solution to \eqref{EqFluidHd0sb}-\eqref{EikonalInsb} in the sense of Definiton \ref{WeakSubSolSystem}. Moreover, under the condition \eqref{Ourassumption}, system \eqref{EqFluidHd0sb}-\eqref{EikonalInsb} admits a unique weak solution. In particular,  
 \begin{enumerate}
\item  If $\beta=2$ and ${\rho \le \frac{\rho_{m}}{2}}$, there exists a unique weak solution of the system \eqref{EqFluidHd0sb}-\eqref{EikonalInsb}.
\item If $\beta=0$, the monotonicity condition \eqref{Ourassumption} is not fulfilled.
\end{enumerate}
\end{thm}
\begin{proof}
Since, we have $p \longmapsto  \mathcal{H}(x, \rho,p)$ is convex, then by weak lower semi-continuity we deduce that $\phi$ verifies 
\begin{align}
-\partial_t\phi  + \mathcal{H}(x, \rho,\nabla \phi ) \le 0, &  \label{EqEikonalHd0sbx}\\
\phi(T)\le\phi_{T}(x).& \label{EikonalInsbx}
\end{align}
By Theorem \ref{Thmlim}, we have $\rho^{\sigma}\mathcal{H}_{p}(x, \rho^{\sigma},\nabla \phi^{\sigma})$ is equi-integrable and therefore weakly converges to $w$  in $L^{1}(\mathrm{Q}_{T})$. 
From the definition of function $\Phi$  \eqref{functionPsi}
Then, we have 
\[
\int_{0}^{T}\int_{\Omega}\Phi(w,\rho) dx dt \le \liminf_{\sigma \to 0}\int_{0}^{T}\int_{\Omega}\Phi(w,\rho) dx dt = \liminf_{\sigma \to 0}\int_{0}^{T}\int_{\Omega} \rho^{\sigma}|\mathcal{H}_{p}(x, \rho^{\sigma},\nabla \phi^{\sigma})|^{2}dx dt \le C,
\]
hence, we have $\Phi(w,\rho) \in L^{1}(\mathrm{Q}_{T})$. In particular we can set $b:= \frac{w}{\rho}1_{\rho>0}$, then $\rho$ is a weak solution of 
\begin{align}
\partial_t\rho  + div(b\rho)=0, &\qquad (t,x)\in\mathrm{Q}_{T}  \label{KOlm1}\\
\rho(t=0,x) =\rho_{0}(x),&\quad x\in\Omega, \label{kolm2}
\end{align}
with $\rho|b|^{2}\in L^{1}(\mathrm{Q}_{T})$. Collecting all the above properties, we get the following inequality. 
\begin{equation}\label{Myeq}
\int_{\Omega}\rho_{0}\phi(0) dx \le \int_{\Omega}\phi_{T}\rho(T) dx + \int_{0}^{T}\int_{\Omega} \rho \left(  b\cdot \nabla \phi - \mathcal{H}(x, \rho,\nabla \phi)\right)
\end{equation}
From system \eqref{EqFluidHd0x}-\eqref{EikonalInx}, we get
\begin{equation}
\int_{\Omega} \phi_{T}(x)\rho^{\sigma}(T) dx + \int_{0}^{T}\rho^{\sigma} \left(  \mathcal{H}_{p}(x, \rho^{\sigma} ,\nabla \phi^{\sigma} )\cdot \nabla \phi^{\sigma}  -\mathcal{H}(x, \rho^{\sigma},\nabla \phi^{\sigma} )\right) = \int_{\Omega} \phi^{\sigma}(0)\rho_{0} dx 
\end{equation}
We observe that $\phi^{\sigma}(0)$ is equi-integrable.  In addition, let $k>0$, $\left(\rho^{\sigma} -k \right)_{+}$ be a sub-solution of the HJB-type equation, then 
\begin{equation}
\begin{aligned}
\int_{\Omega} \left( \phi^{\sigma}(0) - k\right)_{+} dx &+ \int_{0}^{T}\int_{\Omega}\mathcal{H}(x, \rho^{\sigma} ,\nabla \phi^{\sigma} ) 1_{\phi^{\sigma}>k}\\&
 \le \int_{\Omega} \left( \phi_{T} - k \right)_{+} dx 
\end{aligned}
\end{equation}
which implies from the integrability of $\phi_{T}$ that $\int_{\Omega} \left( \phi^{\sigma}(0) - k\right)_{+} dx \to 0$ as $k \to +\infty$ uniformly with respect to $\sigma$. Then, $\phi^{\sigma}(0)$ is weakly converges in $L^{1}(\mathrm{Q}_{T})$ to some function $\tilde{\phi}_{0}$. Therefore we pass to the limit in HJB-type equation \eqref{EqEikonalHd0x} we obtain 
\[
\int_{\Omega}\varphi(0)\tilde{\phi}_{0} dx + \int_{0}^{T}\int_{\Omega} \phi \partial_{t}\varphi dx + \int_{0}^{T}\int_{\Omega} \mathcal{H}(x, \rho ,\nabla \phi)\varphi \le \int_{\Omega}\phi_{T}\varphi(T), \quad \forall \varphi \in C^{1}(\overline{\mathrm{Q}_{T}}), \,\, \varphi \ge 0.
\]

Then, taking a sequence $\varphi_{i}$ such that $\varphi_{i}(0)=1$ and $\varphi_{i}$ approximate the Dirac mass at $t=0$ we deduce that $\tilde{\phi}_{0} \le \phi(0)$ where $\phi(0)$ is the trace of $\phi$ at time $t=0$.

\begin{equation}
\begin{aligned}
\int_{\Omega}\phi_{T}\rho^{\sigma}(T) dx &+ \int_{0}^{T}\int_{\Omega} \rho^{\sigma}\left( \mathcal{H}_{p}(x, \rho^{\sigma} ,\nabla \phi^{\sigma})\cdot \nabla \phi^{\sigma} -\mathcal{H}(x, \rho^{\sigma} ,\nabla \phi^{\sigma}) \right)dx dt \\ & \int_{\Omega}\rho_{0}\tilde{\phi}_{0} dx \le \int_{\Omega}\rho_{0}\phi(0)dx 
\end{aligned}
\end{equation}
Now, using \eqref{Myeq}, we get
\begin{equation}\label{bound0}
\begin{aligned}
\limsup_{\sigma \to 0} &\Big (\int_{\Omega}\phi_{T}\rho^{\sigma}(T) dx + \int_{0}^{T}\int_{\Omega} \rho^{\sigma}\left( \mathcal{H}_{p}(x, \rho^{\sigma} ,\nabla \phi^{\sigma})\cdot\nabla \phi^{\sigma} - \mathcal{H}(x, \rho^{\sigma} ,\nabla \phi^{\sigma}) \right)dx dt \Big)\\ & \le \int_{\Omega}\phi_{T}\rho(T) dx  +  \int_{0}^{T}\int_{\Omega} \rho \left( b\cdot \nabla\phi - \mathcal{H}(x, \rho ,\nabla \phi)  \right) dx dt 
\end{aligned}
\end{equation}
We denote  $w^{\sigma}:= \rho^{\sigma} \mathcal{H}_{p}(x, \rho^{\sigma} ,\nabla \phi^{\sigma})$ where $w$ its weak limit in $L^{1}(\mathrm{Q}_{T})$. Moreover, we have the following
\begin{equation}
\begin{aligned}
 \int_{0}^{T}\int_{\Omega} \rho^{\sigma}&\left( \mathcal{H}_{p}(x, \rho^{\sigma} ,\nabla \phi^{\sigma})\cdot\nabla \phi^{\sigma} - \mathcal{H}(x, \rho^{\sigma} ,\nabla \phi^{\sigma}) \right)dx dt \\&=  \int_{0}^{T}\int_{\Omega}\left(\frac{1}{2}\rho^{\sigma} f^{\beta}(\rho^{\sigma})|\nabla \phi^{\sigma}|^{2} + \frac{1}{2} \rho^{\sigma} f^{2-\beta}(\rho^{\sigma})\right) dx dt \\&=  \int_{0}^{T}\int_{\Omega} \rho^{\sigma} \mathcal{H}^{\ast}(x, \rho^{\sigma} ,\mathcal{H}_{p}(x, \rho^{\sigma} ,\nabla \phi^{\sigma})) dx dt \\&= \int_{0}^{T}\int_{\Omega} \rho^{\sigma} \mathcal{H}^{\ast}(x, \rho^{\sigma} ,\frac{w^{\sigma}}{\rho^{\sigma}}) dx dt,
\end{aligned}
\end{equation}
where $\mathcal{H}^{\ast}$ (or  $\mathcal{L}$)  is the Legendre transform of $\mathcal{H}$ as defined in \eqref{EqHamil}-\eqref{EqHamilY}. 
From Fatou's lemma, we deduce 
\begin{equation}\label{bound}
\begin{aligned}
\liminf_{\sigma \to 0} &\int_{0}^{T}\int_{\Omega} \rho^{\sigma}\left( \mathcal{H}_{p}(x, \rho^{\sigma} ,\nabla \phi^{\sigma})\cdot\nabla \phi^{\sigma} - \mathcal{H}(x, \rho^{\sigma} ,\nabla \phi^{\sigma}) \right)dx dt \\ &=\liminf_{\sigma \to 0} \int_{0}^{T}\int_{\Omega} \left(\frac{1}{2}\rho^{\sigma} f^{\beta}(\rho^{\sigma})|\nabla \phi^{\sigma}|^{2} + \frac{1}{2} \rho^{\sigma} f^{2-\beta}(\rho^{\sigma})\right) dx dt  \\&  \ge  \int_{0}^{T}\int_{\Omega}\left(\frac{1}{2}\rho f^{\beta}(\rho)|\nabla \phi|^{2} + \frac{1}{2} \rho f^{2-\beta}(\rho)\right) dx dt =\int_{0}^{T}\int_{\Omega} \rho  \mathcal{H}^{\ast}(x, \rho ,\frac{w}{\rho})  dx dt .
\end{aligned}
\end{equation}
We have 
\[
 \int_{0}^{T} \int_{\Omega} \rho\left(b\cdot\nabla \phi - \mathcal{H}(x, \rho ,\nabla \phi) \right) dx dt  \le \liminf_{\sigma \to 0} \int_{0}^{T}\int_{\Omega} \rho^{\sigma}\left( \mathcal{H}_{p}(x, \rho^{\sigma} ,\nabla \phi^{\sigma})\cdot\nabla \phi^{\sigma} - \mathcal{H}(x, \rho^{\sigma} ,\nabla \phi^{\sigma}) \right)dx dt,
\]
then we can deduce that 
\[
\int_{0}^{T} \int_{\Omega} \rho\left(b\cdot\nabla \phi - \mathcal{H}(x, \rho ,\nabla \phi) \right) dx dt - \int_{0}^{T}\int_{\Omega} \rho  \mathcal{H}^{\ast}(x, \rho ,\frac{w}{\rho})  dx dt  \le 0.
\]
Therefore, from \ref{bound0} and \ref{bound} we obtain
\begin{equation}\label{Eqss}
\begin{aligned}
\limsup_{\sigma \to 0}&\int_{\Omega}\phi_{T}\rho^{\sigma}(T) dx  \le \int_{\Omega}\phi_{T}\rho(T) dx + \\& \int_{0}^{T} \int_{\Omega} \rho\left(b\cdot\nabla \phi - \mathcal{H}(x, \rho ,\nabla \phi) \right) dx dt - \int_{0}^{T}\int_{\Omega} \rho  \mathcal{H}^{\ast}(x, \rho ,\frac{w}{\rho})  dx dt \\&\le  \int_{\Omega}\phi_{T}\rho(T) dx,
\end{aligned}
\end{equation}
it follows that the limit of the right-hand side and the left-hand side coincides, then we conclude 
\begin{equation}
 \int_{0}^{T} \int_{\Omega} \rho\left(\frac{w}{\rho}\cdot\nabla \phi - \mathcal{H}(x, \rho ,\nabla \phi)  - \mathcal{H}^{\ast}(x, \rho ,\frac{w}{\rho})  \right) dx dt =0,
\end{equation}
We have
\begin{equation}
\begin{aligned}
 \rho&\left(\frac{w}{\rho}\cdot\nabla \phi - \mathcal{H}(x, \rho ,\nabla \phi)  - \mathcal{H}^{\ast}(x, \rho ,\frac{w}{\rho})  \right) \\&=\left(w\cdot\nabla \phi  - \frac{1}{2}\rho f^{\beta}(\rho)|\nabla \phi|^{2} + \frac{1}{2} \rho f^{2-\beta}(\rho) - \frac{1}{2}\rho f^{\beta}(\rho)|\nabla \phi|^{2} - \frac{1}{2} \rho f^{2-\beta}(\rho)\right)\\&=\left(w\cdot\nabla \phi -\rho f^{\beta}(\rho)|\nabla \phi|^{2} \right),
\end{aligned}
\end{equation}
then 
\begin{equation}
 \int_{0}^{T} \int_{\Omega} \left( w\cdot\nabla \phi -\rho f^{\beta}(\rho)|\nabla \phi|^{2}  \right) dx dt =0,
\end{equation}
which implies that $w=\rho f^{\beta}(\rho)\nabla \phi =\rho \mathcal{H}_{p}(x, \rho ,\nabla \phi)$. Therefore, we have \eqref{bound} holds when $w=\rho \mathcal{H}_{p}(x, \rho ,\nabla \phi)$. then ,we have 
\begin{equation}\label{boundw}
\begin{aligned}
\int_{\Omega}\phi_{T}\rho(T) dx &+ \int_{0}^{T}\int_{\Omega}\rho\left( \mathcal{H}_{p}(x, \rho ,\nabla \phi)\cdot\nabla \phi - \mathcal{H}(x, \rho ,\nabla \phi) \right)dx dt \\ & \le \liminf_{\sigma \to 0} \Big (\int_{\Omega}\phi_{T}\rho^{\sigma}(T) dx + \int_{0}^{T}\int_{\Omega} \rho^{\sigma}\left( \mathcal{H}_{p}(x, \rho^{\sigma} ,\nabla \phi^{\sigma})\cdot\nabla \phi^{\sigma} - \mathcal{H}(x, \rho^{\sigma} ,\nabla \phi^{\sigma}) \right)dx dt \Big)\\ & \le \int_{\Omega}\phi(T)\rho(T) dx. 
\end{aligned}
\end{equation}
Combining  the previous inequality with \eqref{Myeq} its yields the energy equality \eqref{WeakfomulationF} and we deduce that $(\rho,\phi)$ is a weak solution of the GH system \eqref{EqFluidHd0sb}\eqref{EikonalInsb} in the sense of Definition \ref{WeakSubSolSystem}. Finally, the proof of the uniqueness results follows from Proposition \ref{propoexun}.
\end{proof}
\begin{remark}

\begin{itemize}
    \item We have performed the analysis above with a specified terminal condition, namely the boundary conditions \eqref{EikonalInx}, which represent a special case of the more general payoff problem. This more general problem employs the following expression:
\[
\phi^{\sigma}(t=T,x)= G(x,\rho^{\sigma}(T)),\quad x\in\Omega,
\]
where $G$ is known as the payoff function. In this context, the Young measures serve as the main tools for proving the existence of the solution, as discussed in\cite{achdou2018mean,achdou2020introduction}.

\item We can show the existence and uniqueness of classical solutions to the system \eqref{EqFluidHd0x}-\eqref{EikonalInx} under the monotonicity condition by using the results of S. Munoz \cite{munoz2022classical,munoz2023classical}. In fact, we assume that the initial data $\rho_{0}$ is in $C^{4}(\Omega)$ and $\rho_{0} \in(0,\rho_{m})$. Then, there exists a unique classical solution $(\rho,\phi)\in C^{3,\alpha}(\bar{\mathrm{Q}_{T}}) \times C^{2,\alpha}(\bar{\mathrm{Q}_{T}})$ to \eqref{EqFluidHd0x}-\eqref{EikonalInx}, for $0<\alpha<1$.
\end{itemize}
\end{remark}

\section{Numerical experiments}\label{sec6}
This section is devoted to the numerical experiments for solving our quasi-stationary fluid \eqref{QuasiStat} by FEM. The computations have been performed with the FreeFem++ library\footnote{https://freefem.org}. It is important to highlight that the presence of viscosity within the model plays a pivotal role in enhancing the numerical stability of our computational scheme. Without viscosity, we would be dealing with a nonlinear hyperbolic system, which falls outside the scope of FreeFem++'s capabilities for solving such complex systems effectively.

For the numerical example, we consider the following domain $\Omega= [0,L]\times[0,R]$ where $L=4$ and $R=2$. A periodic boundary condition for the density $\rho$ in the $x-$direction, that is, $\rho(t,0,y)=\rho(t,L,y)$ for all $y\in [0,R]$, and the y direction we consider the Neumann boundary condition. The viscosity $\sigma \ge 0$ is sufficiently small, $\sigma =0.01$. The time of horizon $T$ is sufficiently large fixed to $T=50$, with a time step of $dt = 5\times10^{-2}$.  We have built triangular meshes $\mathcal{T}_{h}$ of characteristic size $h=1/40$. For initial and boundary data, we have considered an initial pedestrian density $\rho_{0}=\rho_{0i}+\rho_{0p}$, where $\rho_{0i}=0.5$ and $\rho_{0p}$ is a Gaussian distribution localized at the point $(x,y)=(1,1)$ given by 
\[
\rho_{0p}=0.2\times\exp{(-10\times((x-1)^2 +(y-1)^2)}.
\]
$\rho_{0p}$ is the perturbation of the initial data $\rho_{0i}$. We also consider a periodic speed $\nabla \phi$. We assume that $\bar{\rho}=0.5$. Our particular solution of  the GH model is given by 
 \[
 \left(\bar{\rho},\bar{\phi}\right)=\left(\bar{\rho},f^{1-\beta}(\bar{\rho})(L-x)\right), \quad \quad \forall \beta \in [0,2].
 \]
 Therefore, we assume that $\bar{\rho}=0.5$, and 
 \[
 \phi= 0 \qquad (x,y)\in \Sigma\qquad  \phi= f^{1-\beta}(\bar{\rho})L. \qquad (x,y)\in \Gamma.
 \]

Analyzing the Hughes model with $\beta = 2$, as depicted in the velocity fields, see Figures \ref{fsofigbeta20}-\ref{fsofigbeta22}, we observe a clear trend where pedestrians actively avoid high-density regions. The numerical results, illustrated in Figure \ref{fsofigbeta22}, offer insights into the long-time behavior of the model, proving the presence of dispersion and diffusion phenomena over time.

In the scenario where $\beta = 1$, pedestrians are indifferent to the density of the regions because their objective is to reach a specific target. This situation is analogous to the case of pedestrain want to be close to \textit{Hajar Aswad} as explained in section \ref{sec21}, see Figures \ref{fsofigbeta10}-\ref{fsofigbeta12}.

However, when we set $\beta = 0$, we observe a different behavior: pedestrians tend to move toward densely populated areas, a behavior that contradicts the classical Hughes model where individuals typically steer clear of high-density regions. The formation of concentration phenomena becomes evident, as seen in Figures \ref{fsofigbeta00}-\ref{fsofigbeta02}. Figure \ref{fsofigbeta02} displays results up to a horizon time of $T=13.5$. Beyond this point, we begin to notice some instability in the model, possibly arising from the concentration phenomena within the system or due to the sensitivty of numerical scheme to small viscosity values in the FreeFem++ code.


In our concluding remarks, it becomes evident that our numerical simulations not only confirm the validity of the Hughes model, but also shed light on the pivotal role that the parameter $\beta$ can play as a safety measure within the model, contributing to the effective regulation of pedestrian evacuation dynamics. This dual outcome emphasizes the capacity of the model to not only replicate real-world pedestrian behavior but also to serve as a versatile tool for enhancing safety protocols in crowded environments.


Finally, our work opens the door to exploration of stability/instability and long-time behavior within the MFG system. In our forthcoming paper \cite{GMP2023preprint}, we present a comprehensive analysis of the stability of the GH model. 

\begin{figure}[h]
	\begin{tabular}{cc}
		\includegraphics[scale=0.2]{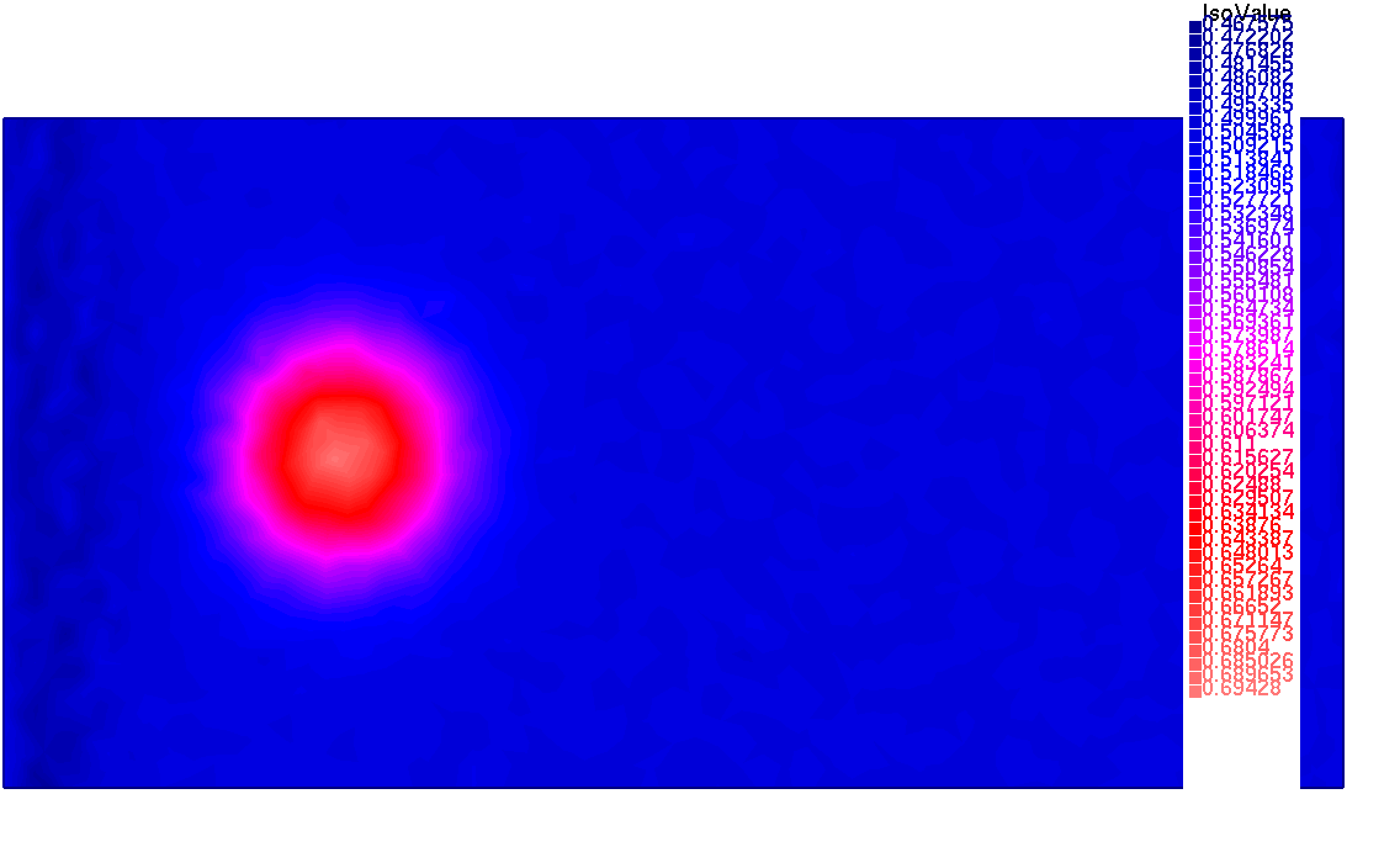}&\includegraphics[scale=0.2]{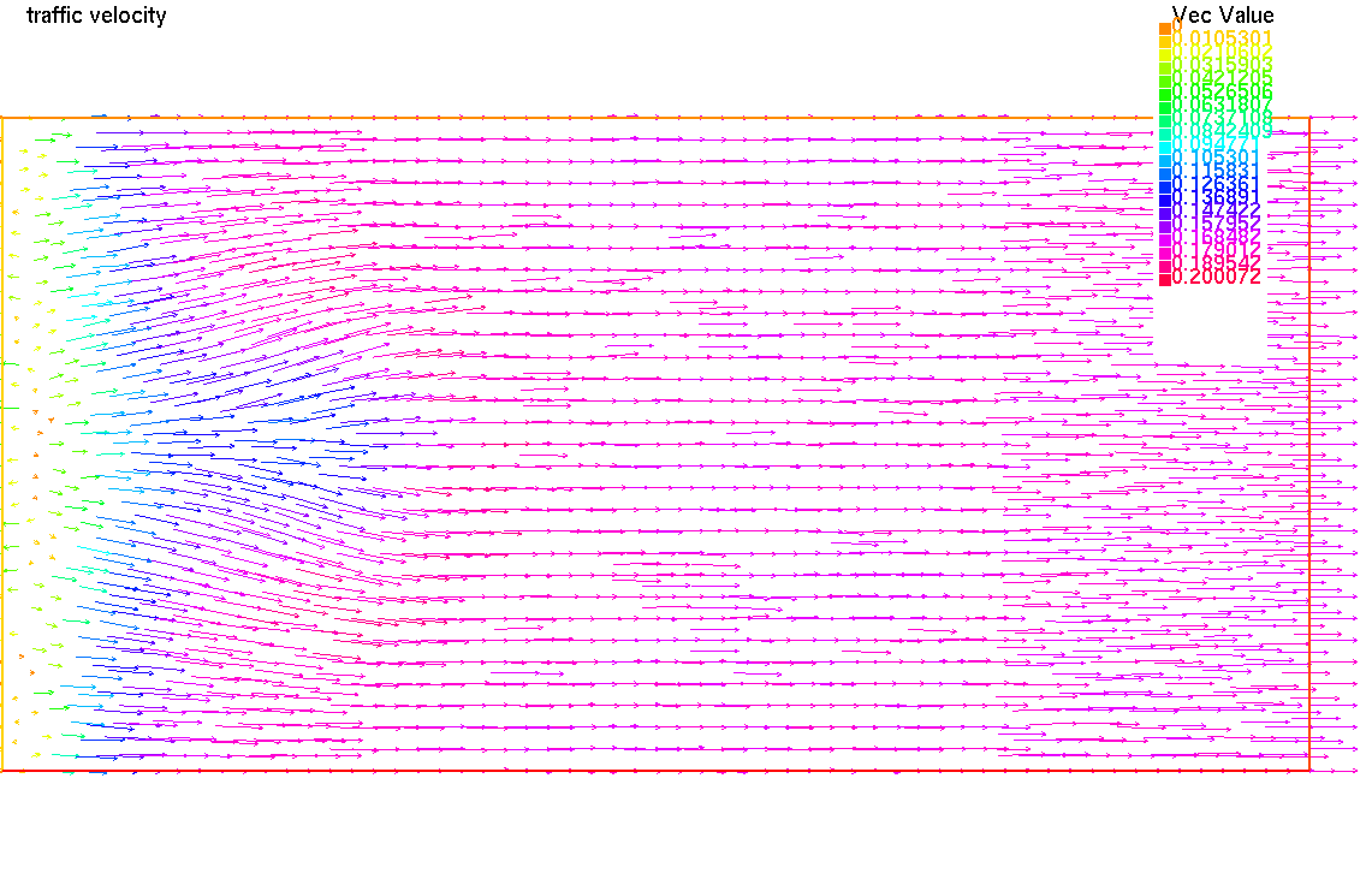}\\
		(a) & (b)
	\end{tabular}
	\caption{Schematic illustration of  density contour levels (a) and velocity field (b) at the initial instant $t=0$ ($\beta=2$). }
	\label{fsofigbeta20}
\end{figure}

\begin{figure}[!h]
	\begin{tabular}{cc}
		\includegraphics[scale=0.25]{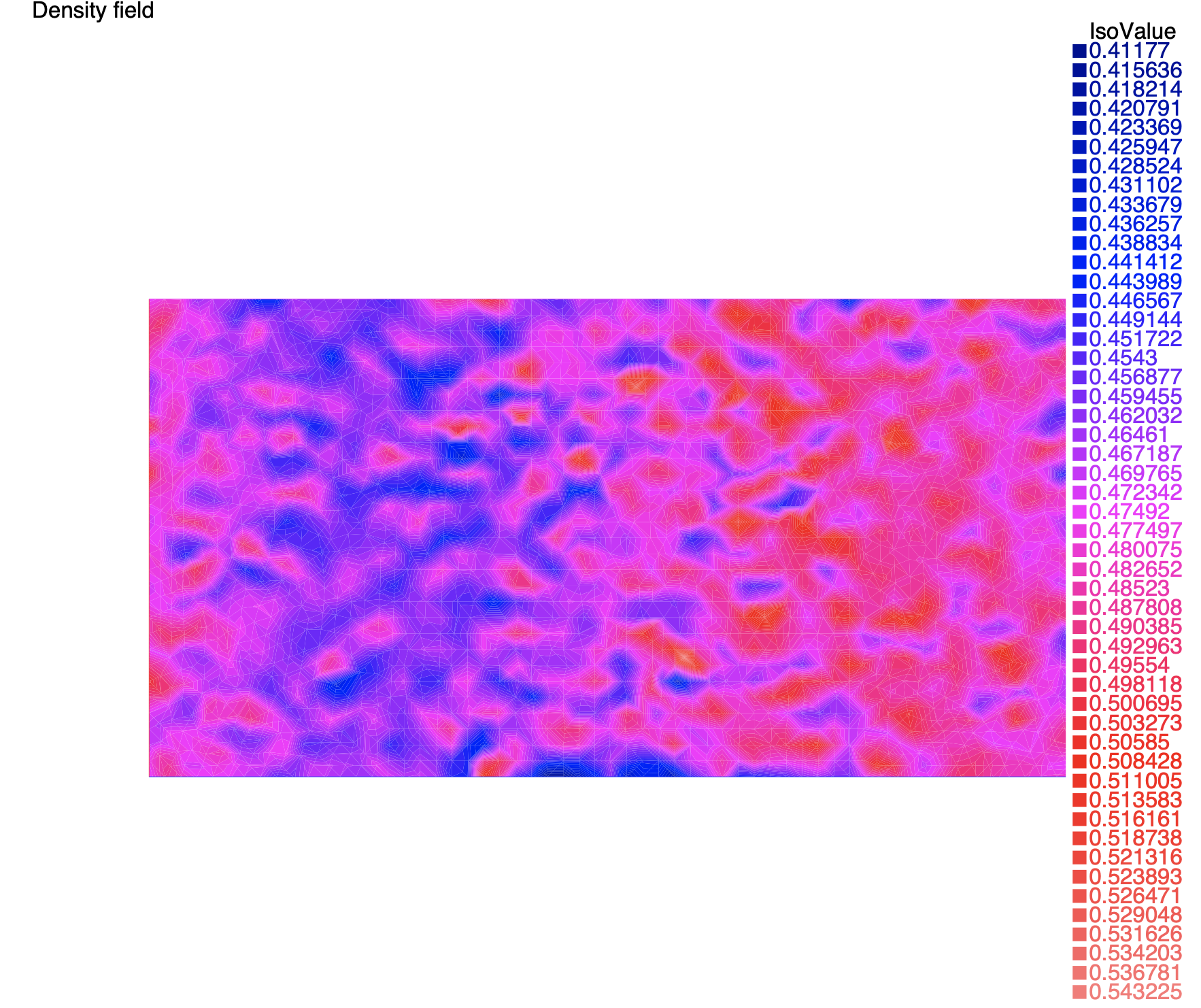}&\includegraphics[scale=0.25]{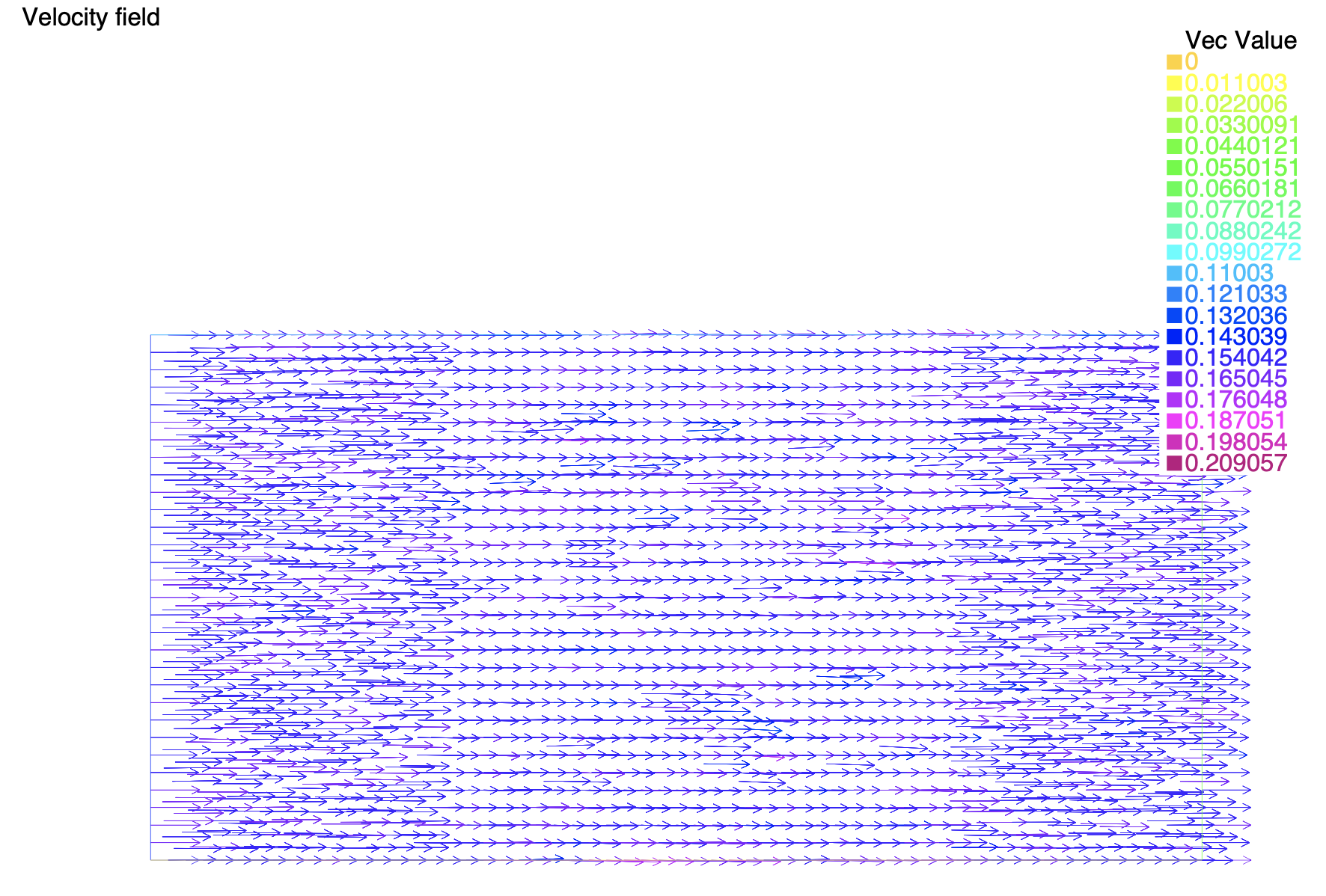}\\
		(a) & (b)
	\end{tabular}
	\caption{Schematic illustration of  density contour levels (a) and velocity field (b) at  final time  $T$ and   $\beta=2$. }
	\label{fsofigbeta22}
\end{figure}
\begin{figure}
	\begin{tabular}{cc}
		\includegraphics[scale=0.2]{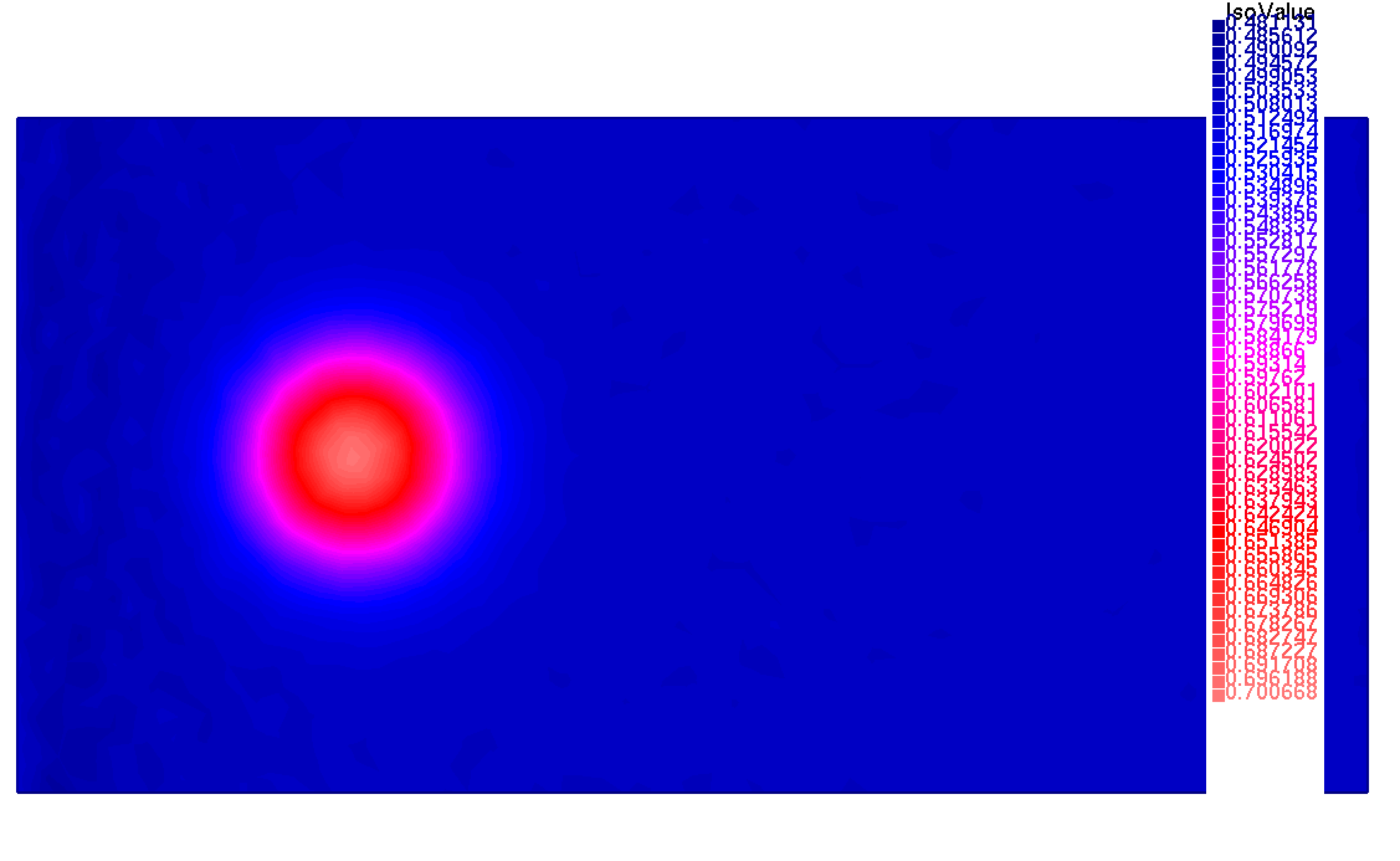}&\includegraphics[scale=0.2]{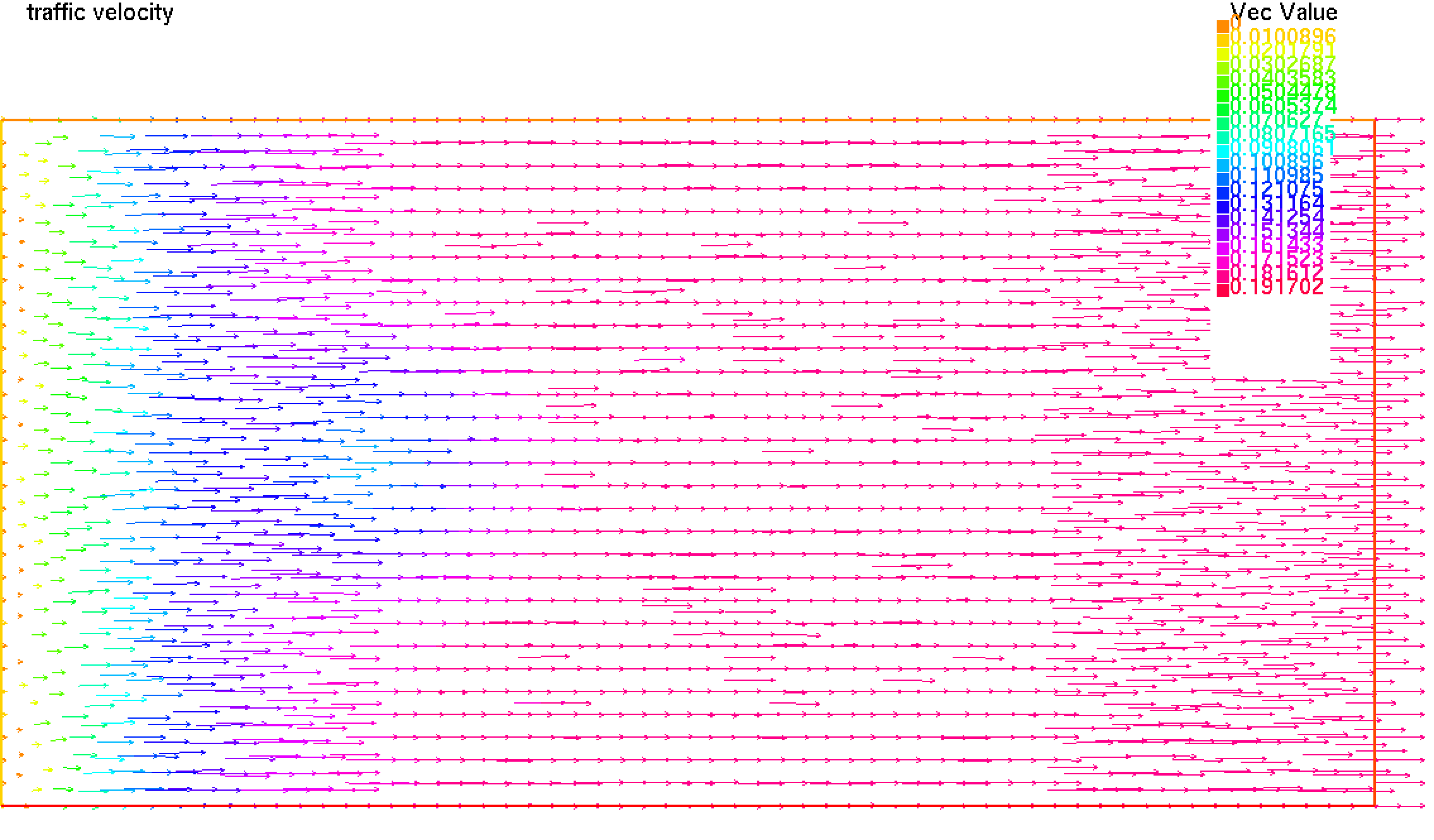}\\
		(a) & (b)
	\end{tabular}
	\caption{Schematic illustration of  density contour levels (a) and velocity field (b) at the initial instant $t=0$ ($\beta=1$). }
	\label{fsofigbeta10}
\end{figure}
%

\begin{figure}
	\begin{tabular}{cc}
		\includegraphics[scale=0.25]{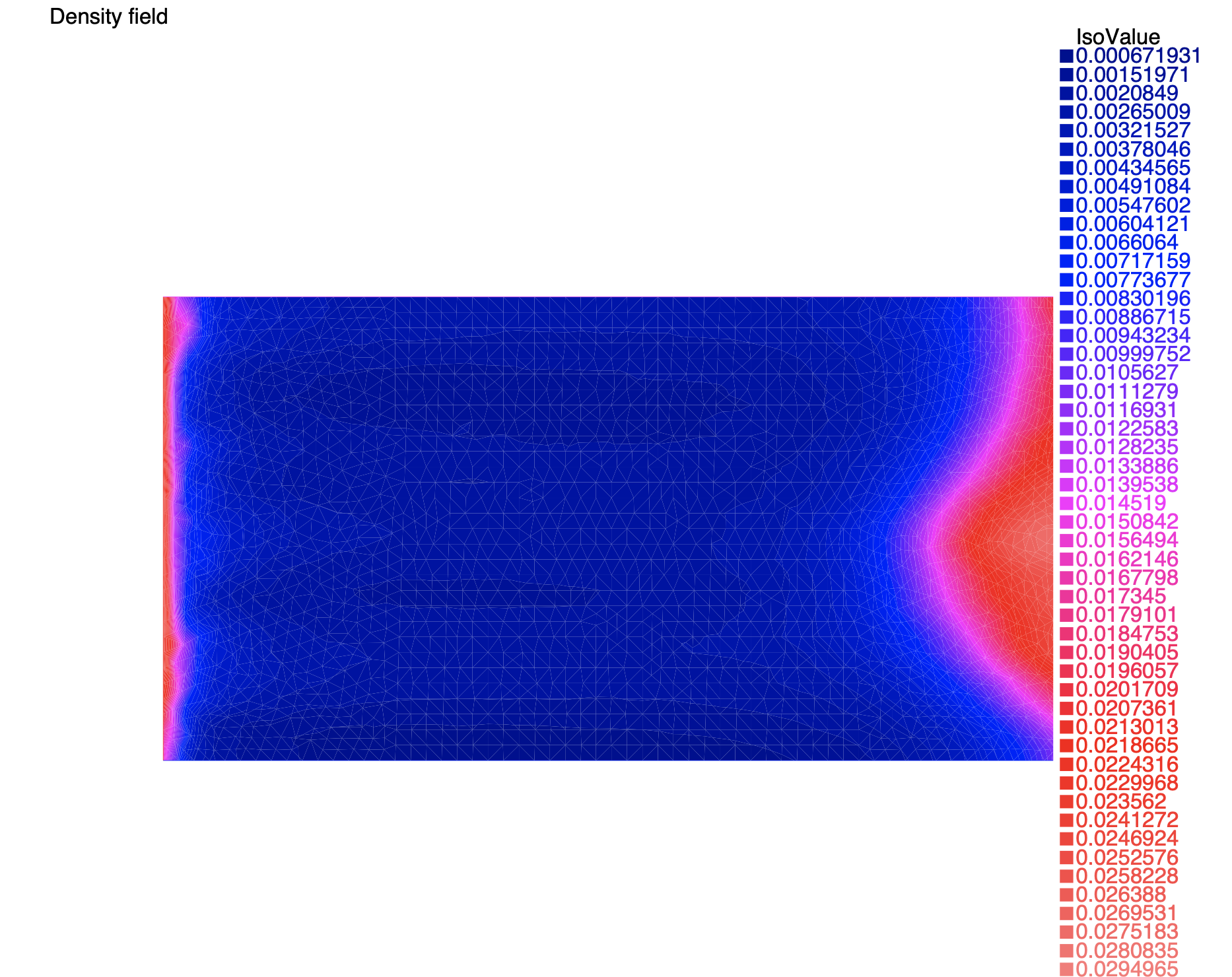}&\includegraphics[scale=0.25]{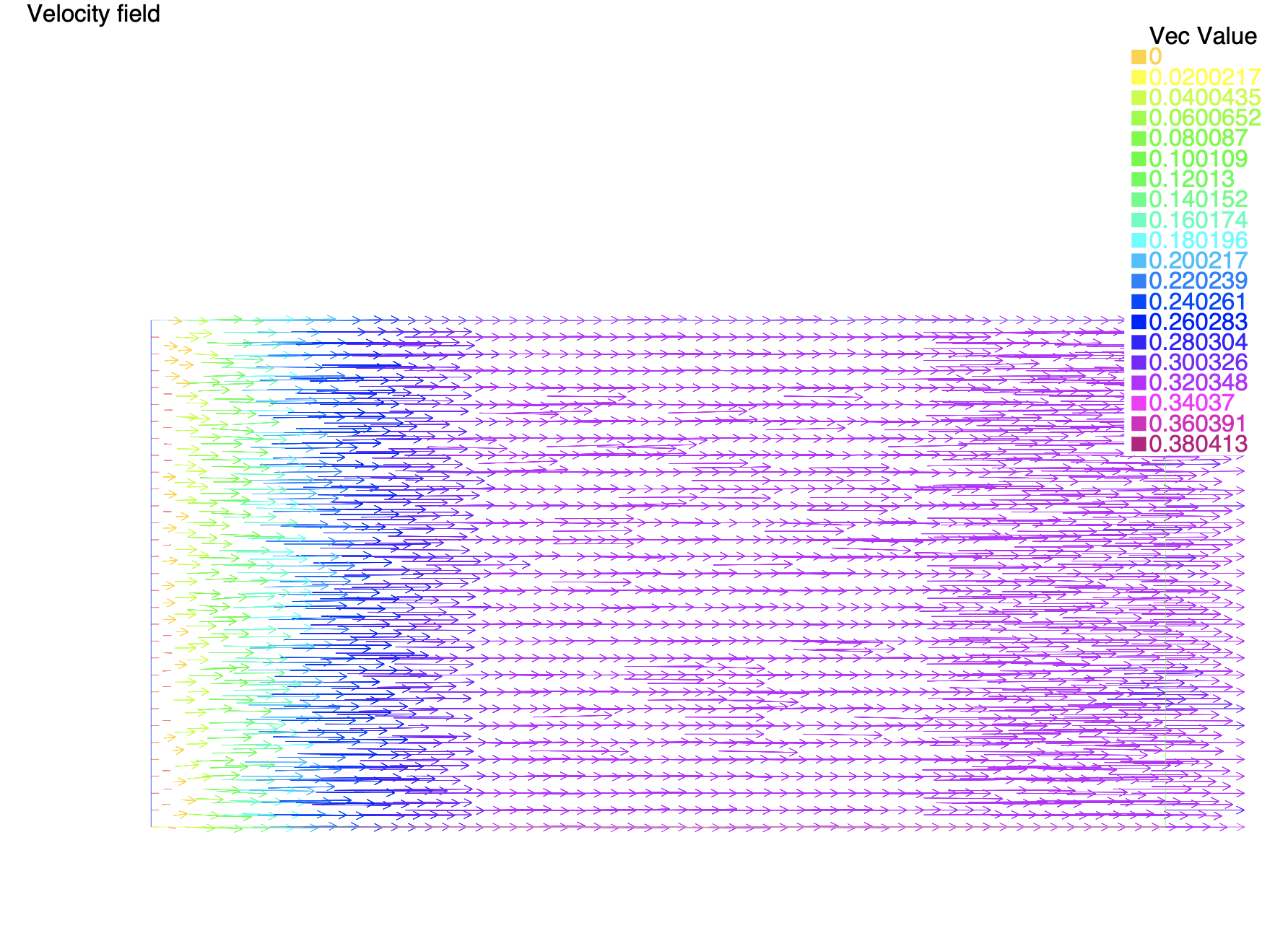}\\
		(a) & (b)
	\end{tabular}
	\caption{Schematic illustration of  density contour levels (a) and velocity field (b) at  final time  $T$ and   $\beta=1$. }
		\label{fsofigbeta12}
\end{figure}
\begin{figure}[!h]
	\begin{tabular}{cc}
		\includegraphics[scale=0.2]{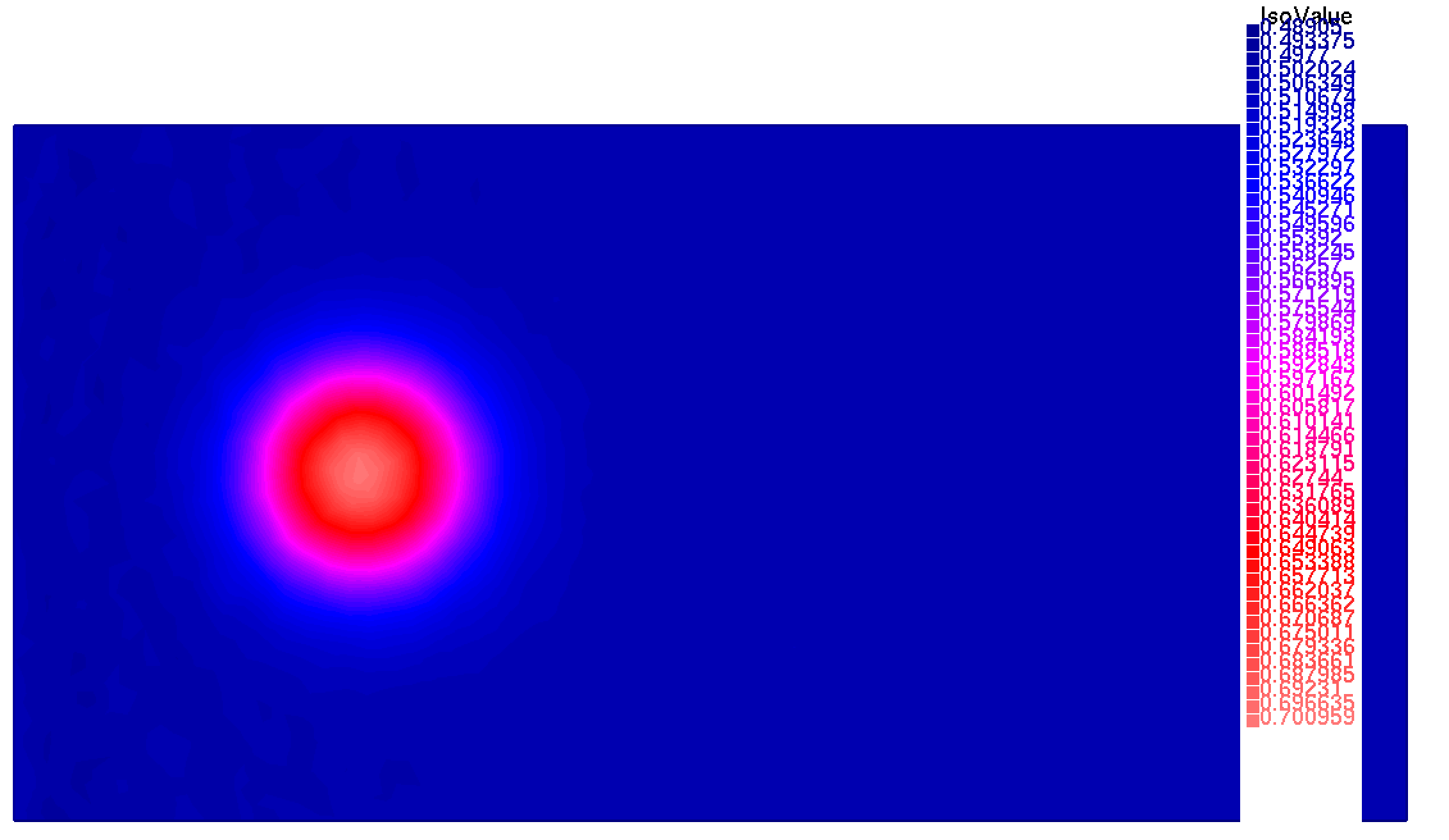}&\includegraphics[scale=0.2]{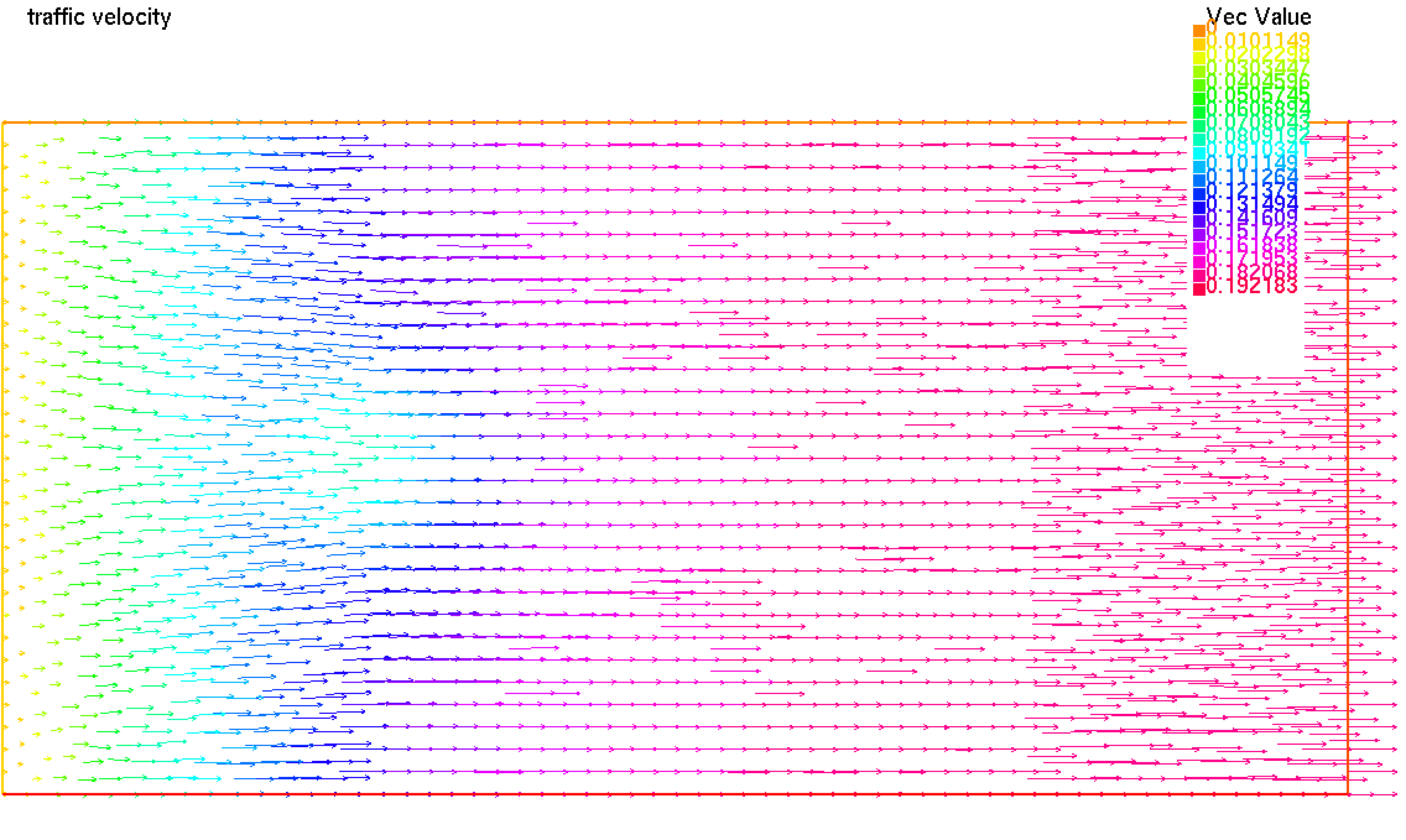}\\
		(a) & (b)
	\end{tabular}
	\caption{Schematic illustration of  density contour levels (a) and velocity field (b) at the initial instant $t=0$ and   $\beta=0$. }
	\label{fsofigbeta00}
\end{figure}

\begin{figure}[!h]
	\begin{tabular}{cc}
		\includegraphics[scale=0.25]{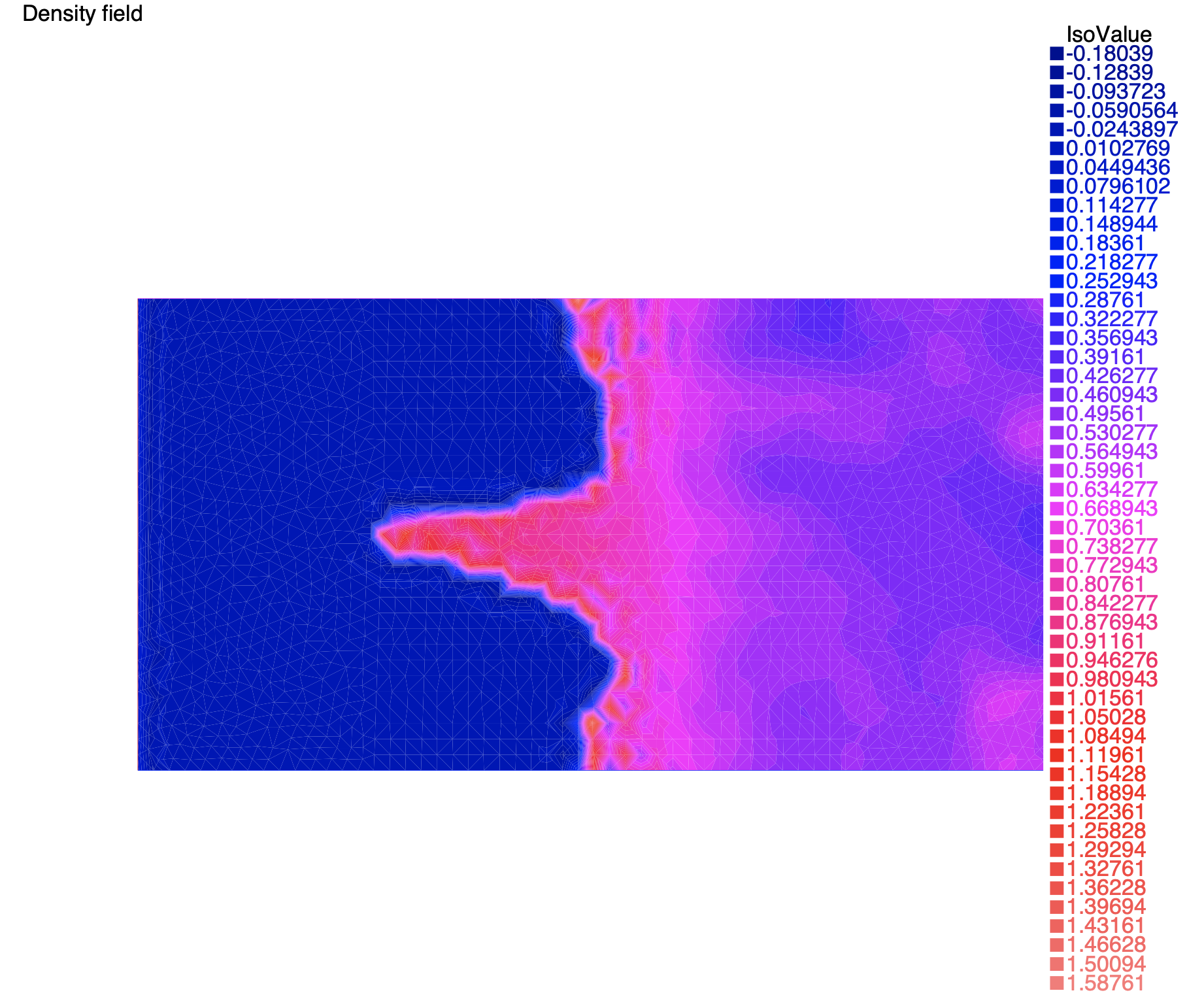}&\includegraphics[scale=0.25]{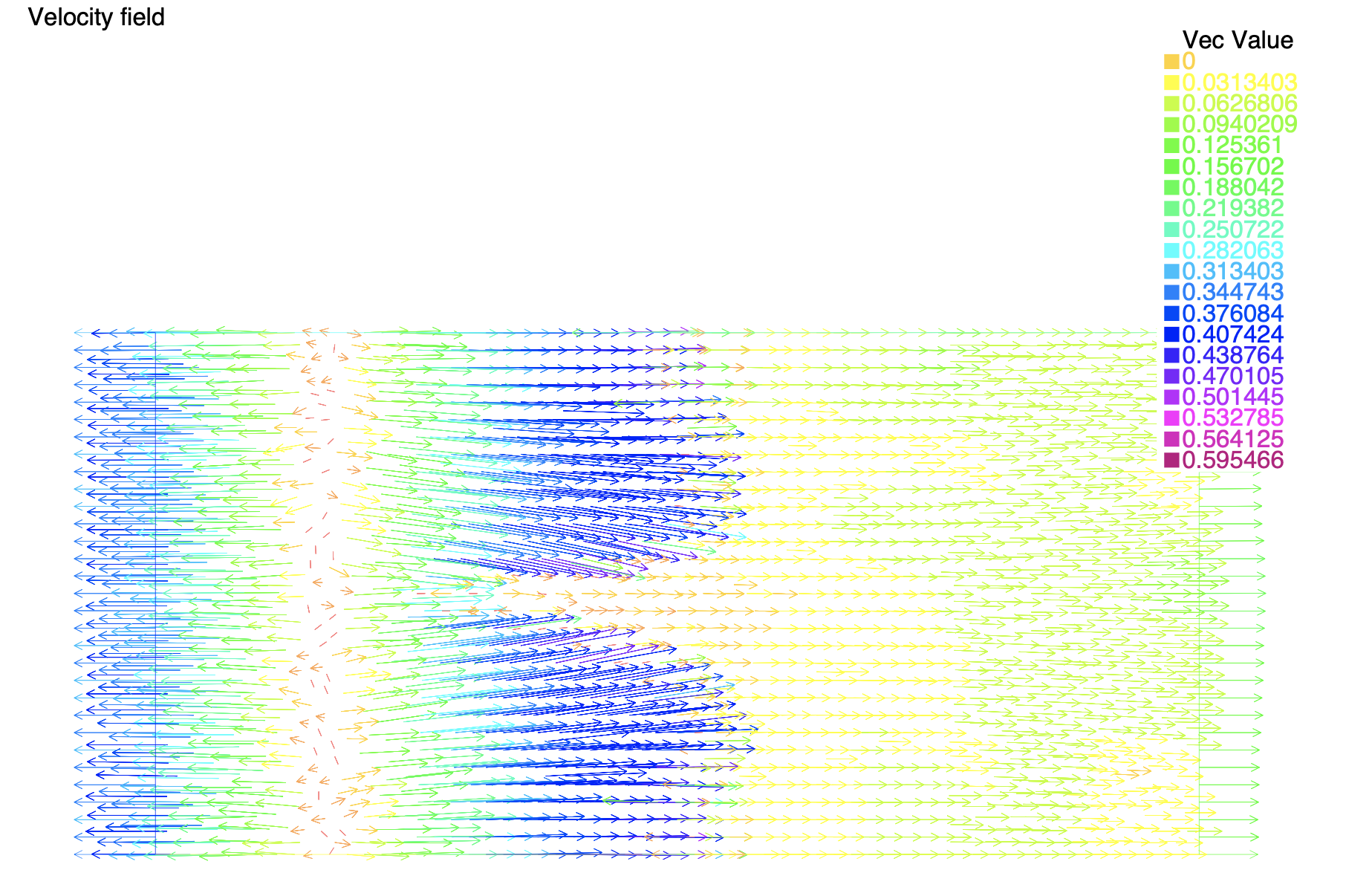}\\
		(a) & (b)
	\end{tabular}
	\caption{Schematic illustration of  density contour levels (a) and velocity field (b) at time  $t=13.5$ and   $\beta=0$. }
	\label{fsofigbeta02}
\end{figure}

\section{Acknowledgments}
The work of the authors is supported by Tamkeen under the NYU Abu Dhabi Research Institute grant of the center SITE. M.G. and N.M. thank Prof. Diogo Gomes (KAUST) for discussing the problem during his visit to NYU Abu Dhabi. Special thanks are extended to Prof. Diogo Gomes (KAUST) and the entire KAUST team for their warm hospitality, greatly appreciated by M.G. Furthermore, M.G. expresses gratitude to Prof. Alessio Porretta (Università di Roma Tor Vergata) for bringing the reference \cite{achdou2020introduction} to our attention.

\bibliographystyle{unsrt}
\bibliography{HughesmodelingRefs}

\end{document}